\documentclass[11pt,reqno,final]{article}

\usepackage[color,notref,notcite]{showkeys}
\definecolor{labelkey}{rgb}{0.6,0,1}
\usepackage{epsfig}
\usepackage{a4wide,amssymb,amsmath}
\usepackage{amsthm}
\usepackage{color}
\usepackage{url}
\usepackage{mathtools}

\usepackage{graphicx}

\usepackage[english]{babel}
\usepackage[utf8]{inputenc}
\usepackage[T1]{fontenc}
\usepackage{fancyhdr}
\usepackage{dsfont}
\usepackage{indentfirst}
\usepackage{color}
\usepackage{newlfont}
\usepackage{float}
\usepackage{amssymb}
\usepackage{latexsym}
\usepackage{mathrsfs}
\usepackage{listings}
\usepackage{breqn}
\usepackage{caption}
\usepackage{subcaption}
\usepackage{enumerate}
\usepackage{hyperref}

\newcommand{\ep}{\varepsilon}

\newcommand{\bd}{\begin{displaymath}}

\newcommand{\ed}{\end{displaymath}}
\newcommand{\be}{\begin{eqnarray}}
\newcommand{\ee}{\end{eqnarray}}

\newcommand{\ba}{\begin{array}}
	\newcommand{\ea}{\end{array}}

\newtheorem{theorem}{\bf Theorem}[section]
\newtheorem{proposition}[theorem]{\bf Proposition}

\newtheorem{definition}[theorem]{\bf Definition}

\newtheorem{example}[theorem]{Example}

\newtheorem{lemma}{\textbf{Lemma}}[section]
\newtheorem{remark}[theorem]{\bf Remark}
\providecommand{\keywords}[1]{\textbf{\textit{Keywords: }} #1}

\begin{document}

\title{Convergence of Adaptive Filtered Schemes for First Order Evolutionary Hamilton-Jacobi Equations
\thanks{The three authors are members of the INdAM Research Group GNCS.}
}

\author{Maurizio Falcone\footnotemark[1] \and Giulio Paolucci\footnotemark[1] \and Silvia Tozza\footnotemark[2]}

\maketitle

\renewcommand{\thefootnote}{\fnsymbol{footnote}}

\footnotetext[1]
{
	Dipartimento di Matematica,  ``Sapienza" Universit{\`a}  di Roma,
	P.le Aldo Moro, 5 - 00185 Rome, Italy
	({\tt e-mail:falcone@mat.uniroma1.it,paolucci@mat.uniroma1.it})
}

%\footnotetext[2]
%{
%	Dipartimento di Matematica,  ``Sapienza" Universit{\`a}  di Roma,
%	P.le Aldo Moro, 5 - 00185 Rome, Italy
%	({\tt e-mail: paolucci@mat.uniroma1.it})
%}

\footnotetext[2]
{
	Dipartimento di Matematica e Applicazioni “Renato Caccioppoli”,  Universit{\`a} degli Studi di Napoli Federico II, 
	Via Cintia, Monte S. Angelo, I-80126 Napoli, Italy 
	({\tt e-mail: silvia.tozza@unina.it})
}

\begin{abstract}
We consider a class of ``filtered'' schemes for first order time dependent Hamilton-Jacobi equations and prove a general convergence result for this class of schemes. A typical  filtered scheme is obtained mixing a high-order scheme and a monotone scheme  according to a filter function $F$ which decides where the scheme has to switch from one scheme to the other. A crucial role for this switch is played by a parameter $\varepsilon=\varepsilon({\Delta t,\Delta x})>0$ which  goes to 0 as the time and space steps $(\Delta t,\Delta x)$ are going to 0 and  does not depend on the time $t_n$, for each iteration $n$. The tuning of this parameter in the code is rather delicate and has an influence on the global accuracy of the filtered scheme. 
Here  we  introduce an adaptive and automatic choice of  $\varepsilon=\varepsilon ^n (\Delta t, \Delta x)$ at every iteration modifying the classical set up. The adaptivity is controlled by  a smoothness indicator which selects  the regions where we modify the regularity threshold $\varepsilon^n$.
A convergence result and some error estimates for the new adaptive filtered scheme are proved, this analysis relies on the properties of the scheme and of the smoothness indicators. 
Finally,  we  present  some numerical tests to compare the  adaptive filtered scheme with other methods. \\

\noindent \keywords{High-order Filtered schemes, Hamilton-Jacobi equations, Convergence, Smoothness indicators}
\end{abstract}

%%%%%%%%%%%%%%%%%%%%%%%%%%%%%%%%%%%%%%%%%%%%%%
%%%%%%%%%%%%%%%  SECTION 1: introduction
%%%%%%%%%%%%%%%%%%%%%%%%%%%%%%%%%%%%%%%%%%%%%%
\section{Introduction}
\label{intro}
Here we propose and analyze a new adaptive filter scheme and prove its convergence to the viscosity solution of the scalar evolutionary Hamilton-Jacobi equation
\begin{equation}
	\label{HJ}
	\left\{
	\begin{array}{ll}
		v_t+H(v_x)=0,\qquad&(t,x)\in[0,T]\times\mathbb{R}, \\
		v(0,x)=v_0(x),&x \in \mathbb{R},
	\end{array}
	\right.
\end{equation}
where Hamiltonian $H$ and the initial data $v_0$ are Lipschitz continuous functions. A precise result of existence and uniqueness in the framework of weak viscosity solutions can be found in \cite{B94} and the precise setting of assumptions will be given in Sect.  \ref{Sec:AFS}.\\
The accurate numerical solution of Hamilton-Jacobi (HJ) equations is a challenging topic of growing importance in many fields of application, e.g. control theory, KAM theory, image processing and material science. Due to the lack of regularity of viscosity solutions, this issue is delicate and the construction of high-order methods can be  rather  complicated and the proof of convergence is challenging. It is well known that simple monotone schemes are at most first order accurate
as shown in \cite{CL84} so monotonicity should be abandoned to get high-order convergence.
Our goal  is to present a rather simple way to construct convergent schemes to the viscosity solution $v$ of \eqref{HJ} with the property to be of high-order in the region of regularity.\\
In recent years a general approach to the construction of high-order methods using filters has been proposed by Lions and Souganidis in \cite{LS95} and further developed by Oberman and  Salvador \cite{OS15}. 
It is also interesting to mention that filtered schemes were also used for second order problems by Froese and Oberman \cite{FO13} for the Monge-Ampere equation and, more recently, by Bokanowski, Picarelli and Reisinger \cite{BPR18} who studied second order time dependent HJB equations. 
Let us remind that a typical feature of a filtered scheme $S^F$   is that at the node $x_j$  the scheme is  combination of a high-order scheme $S^A$ and a monotone scheme $S^M$ according to a filter function $F$. The scheme is written as 
\begin{equation}
	\label{eq:FS}
	u^{n+1}_j \equiv S^F(u^n)_j := S^{M}(u^n)_j+\varepsilon\Delta t F\bigg(\frac{S^{A}(u^n)_j-S^{M}(u^n)_j}{\varepsilon\Delta t}\bigg), \quad  j\in \mathbb{Z},
\end{equation}
where $u^{n+1}_j:=u(t_{n+1},x_j)$ is the numerical approximation at time $t_{n+1}$ and node $x_j$, $\varepsilon=\varepsilon(\Delta t,\Delta x)>0$  {\em  is a fixed parameter} going to 0 as $(\Delta t,\Delta x)$ is going to 0 and  
does not depend on $n$.
Filtered schemes  are high-order accurate  where the solution is smooth, monotone otherwise, and  this feature is crucial to prove a convergence result as in  \cite{BFS16}.  Note that the choice of the parameter $\varepsilon$ is delicate because it plays a crucial role in the switching
so its tuning is rather important (see  \cite{BFS16} for a detailed discussion of this point). Then it seems natural to adapt its choice to the regularity of the solution in the cell via a smoothness indicator.
Here we improve  the filtered scheme \eqref{eq:FS} introducing an {\em adaptive and automatic choice of the parameter $\varepsilon=\varepsilon ^n$ at every iteration}. To this end, we introduce a smoothness indicator  to select the regions where we have to update the regularity threshold $\varepsilon^n$, this  indicator is  chosen according to the analysis proposed in \cite{JP00} although other proposals with similar properties can be applied. 

To set this paper into perspective let us remind that the construction of high order methods for hyperbolic equations  has been a very active research area started by  the seminal paper \cite{HOEC86}. Several techniques have been proposed to improve the accuracy leading to essentially non oscillatory schemes ENO and weighted ENO (so called WENO) for conservation laws as in \cite{JS96,AB03,HAP05,ABBM11}, for a survey on these high-order  techniques we refer to \cite{S98,ZS16}. More recently a centered  and more efficient version (called CWENO) has been proposed in \cite{CPSV18}. We should also mention that  high-order methods have been proposed for Hamilton-Jacobi either extending the ENO approach as in \cite{JP00,KNP01,CFR05} or by semi-lagrangian techniques as extensively discussed in \cite{FF13}. For a recent survey on the numerical approximation of Hamilton-Jacobi equations we refer the interested reader to \cite{FF16}. 

%%%Organization
The paper is organized as follows:\\
In Sect. \ref{Sec:AFS} we construct  the new adaptive filtered scheme and present in detail all its building blocks, the main assumptions are given there. 
Sect. \ref{sec:indicators} is focused on the analysis of the smoothness indicators in one dimension. In Sect. \ref{sec:convergence} we state and prove the main convergence result (that was announced in \cite{FPT19}), some technical lemmas are proved in Appendix \ref{sec:appendix} at the end of this paper.    Finally in Sect.~\ref{sec:tests} we  present several tests to show the effectiveness of the adaptive scheme with respect to the basic filtered scheme and to other state-of-the-art methods. Sect.~\ref{sec:conclusions} contains the conclusions with final comments. 

%%%%%%%%%%%%%%%%%%%%%%%%%%%%%%%%%%%%%%%%%%%%%%
%%%%%%%%%%%%%%%  SECTION 2
%%%%%%%%%%%%%%%%%%%%%%%%%%%%%%%%%%%%%%%%%%%%%%
\section{A new Adaptive Filtered scheme}
\label{Sec:AFS}
Consider the  first order evolutionary Hamilton-Jacobi equation \eqref{HJ} where the hamiltonian $H$ and the initial data $v_0$ are Lipschitz continuous functions. It is well known that with these assumptions we have the existence and uniqueness of the viscosity solution. Notice that to keep the ideas clear we are considering the most simple scalar case with the hamiltonian depending only on the derivative of the solution, with more general situations following directly.

Our aim is to present a rather simple way to construct convergent schemes to the viscosity solution $v$ of (\ref{HJ}) with the property to be of high-order whenever some regularity is detected.
Starting from the ideas of \cite{BFS16} on \emph{filtered schemes}, we proceed in this study introducing a procedure to compute the regularity threshold $\varepsilon$ in an automatic way, in order to exploit the local regularity of the solution.

Let us begin defining a uniform grid in space $x_j=j\Delta x$, $j\in \mathbb Z$, and in time $t_n=t_0+n\Delta t$, $n\in[0,N] \cap \mathbb{N}$, with $(N-1)\Delta t< T\leq N\Delta t$. Then, we compute the numerical approximation $u^n_j=u(t_n,x_j)$ with the simple formula
\begin{equation}
	\label{AFS}
	u_j^{n+1}=S^{AF}(u^n)_j:=S^M(u^n)_j+\phi_j^n\varepsilon^n\Delta t F\left(\frac{S^A(u^n)_j-S^M(u^n)_j}{\varepsilon^n \Delta t}\right),
\end{equation}
where $u^{n+1}_j:=u(t_{n+1},x_j)$, $S^M$ and $S^A$ are respectively the monotone and the high-order scheme, $F$ is the \emph{filter function} needed to switch between the two schemes, $\varepsilon^n$ is the switching parameter at time $t_n$ and $\phi_j^n$ is the \emph{smoothness indicator function} at the node $x_j$ and time $t_n$. More details on the components of the schemes will be given in the following sections.

Notice that if $\varepsilon^n\equiv \varepsilon\Delta x$, with $\varepsilon>0$ and $\phi_j^n\equiv 1$, we get the Basic Filtered Scheme \eqref{eq:FS}.
%%%%%%%%%%%%%%%%%%%%%%%%%%%%%%%%%%%%%%%%%%%%%%%%%%%%%%%%%%%%%%%%%%
%%%%%%%%%%%%%%%%%%%%%%%%%%%%%%%%%%%%%%%%%%%%%%%%%%%%%%%%%%%%%%%%% 
\subsection{Assumptions on the schemes}
In this section we present the basic components of our scheme, which are a monotone finite difference scheme $S^M$ and a high-order, possibly unstable, scheme $S^A$. Let us begin by giving the assumptions on the monotone scheme.\\

\emph{Assumptions on $S^M$.}
\begin{description}
	\item [(M1)]The scheme can be written in \emph{differenced form}
	\begin{equation*}\label{eq:FD}
		u^{n+1}_j\equiv  S^{M}(u^{n}_j) := u^{n}_j -\Delta t ~ h^M(D^-u^n_j,D^+u^n_j)
	\end{equation*}
	for a function $h^M(p^-,p^+)$, with $D^{\pm} u^n_j:=\pm \frac{u^n_{j\pm 1}-u^n_j}{\Delta x}$;
	\item [(M2)] $h^M$ is a Lipschitz continuous function;
	\item [(M3)] (Consistency) $\forall v$, $h^M(v,v)=H(v )$;
	\item [(M4)] (Monotonicity) for any functions $u, v$ 
	$\quad u \leq v\quad  \Rightarrow\quad S^M(u) \leq S^M(v)$.	
\end{description}

Under assumption (M2), the consistency property (M3) is equivalent to say that for all functions $v\in C^2([0,T]\times\mathbb R)$, there exists a constant $C_M\geq 0$ independent on $\Delta=(\Delta t,\Delta x)$ such that
\begin{equation}
	\label{consistenza_M}
	\mathcal E_{M}(v)(t,x):=\left|\frac{v(t+\Delta t,x)-S^M(v(t,\cdot))(x)}{\Delta t}\right|\leq C_M\left(\Delta t ||v_{tt}||_\infty +\Delta x||v_{xx}||_\infty\right),
\end{equation}
where $\mathcal E_M$ is the consistency error. The last relation clearly shows the bound on the accuracy of the monotone schemes, which are at most first order accurate even for regular solutions.
\begin{remark}
	As pointed out in \cite{BFS16}, under the Lipschitz assumption (M2) the monotonicity property (M4) can be restated in terms of some quantities that can be easily computed. In fact, it is enough to require, for a.e. $(p^-,p^+)\in \mathbb R^2$,
	\begin{equation}
		\frac{\partial h^M}{\partial p^-}(p^-,p^+)\geq 0,\qquad \frac{\partial h^M}{\partial p^+}(p^-,p^+)\leq 0,
	\end{equation}
	and the \emph{CFL condition}
	\begin{equation}
		\label{cond_CFL}
		\frac{\Delta t}{\Delta x}\left(\frac{\partial h^M}{\partial p^-}(p^-,p^+)-\frac{\partial h^M}{\partial p^+}(p^-,p^+)\right)\leq 1.
	\end{equation}
	We call the \emph{CFL number}, dependent on the hamiltonian of the considered problem, the constant ratio $\lambda:=\frac{\Delta t}{\Delta x}$ such that (\ref{cond_CFL}) is satisfied. Notice that working with explicit finite difference schemes this number can always be computed.

\noindent An important consequence of property (M4) is the \emph{nonexpansivity in $L^\infty$} of the mapping $S^M$ (see \cite{CL84}, page 8), that is, for any functions $u,v$,
\begin{equation}
\label{prop_nonexp}
||S^M(u)-S^M(v)||_\infty\leq||u-v||_\infty.
\end{equation}
\end{remark}
\begin{example}
	We give some examples of monotone schemes in differenced form which satisfy (M1)-(M4). Other examples may be found in the pioneering work \cite{CL84} or in \cite{S98}.
	\begin{itemize}
		\item For the \emph{eikonal equation},
		$$
		v_t+|v_x|=0,
		$$
		we can use the simple numerical hamiltonian 
		\begin{equation}
			\label{hm_eik}
			h^M(p^-,p^+):=\max\{p^-,-p^+\}.
		\end{equation}
		\item For general equations, instead, we recall the \emph{Central Upwind scheme} of \cite{KNP01}
		\begin{equation}
			\label{central_upwind}
			h^M(p^-,p^+):=\frac{1}{a^+ - a^-}\left[a^-H(p^+)-a^+H(p^-) -a^+a^-(p^+-p^-)\right], 
		\end{equation}
		with $a^+=\max\{H_p(p^-),H_p(p^+),0\}$ and $a^-=\min\{H_p(p^-),H_p(p^+),0\}$, using the usual notation $H_p$ for the derivative of $H$ with respect to $v_x$.
		\item Another numerical hamiltonian we could use is the \emph{Lax-Friedrichs hamiltonian}
		\begin{equation}\label{eq:LF_hamiltonian}
			h^M(p^-,p^+):=H\left(\frac{p^-+p^-}{2}\right)-\frac{\theta}{2}(p^+-p^-)
		\end{equation}
		where $\theta>0$ is a constant. The scheme is monotone under the restrictions $\max_{p}|H_p(p)|<\theta$ and $\theta \lambda\leq 1$.
	\end{itemize}
\end{example}

Next, we define the requirements on the high-order scheme.\\

\emph{Assumptions on $S^A$.}
\begin{description}
	\item [(A1)] The scheme can be written in \emph{differenced form}
	\begin{equation*}
		\label{eq:HA}
		u^{n+1}_j=S^{A}(u^{n})_j :=u^{n}_j-\Delta t  h^{A}(D^{k,-}u_j,\dots, D^-u^n_j,D^+u^n_j,\dots,D^{k,+}u^n_j),
	\end{equation*}
	for some function $h^A(p^-,p^+)$ (in short), with $D^{k,\pm} u^n_j:=\pm \frac{u^n_{j\pm k}-u^n_j}{k\Delta x}$;
	\item [(A2)] $h^A$ is a Lipschitz continuous function.
	\item [(A3)] (High-order consistency) Fix $k\geq 2$ order of the scheme, then for all $l=1,\dots,k$ and for all functions $v\in C^{l+1}$, there exists a constant $C_{A,l}\geq 0$ such that
	\begin{align*}
		\mathcal E_{A}(v)(t,x):=&\left|\frac{v(t+\Delta t,x)-S^A(v(t,\cdot))(x)}{\Delta t}\right|\\
		&\leq C_{A,l}\left(\Delta t^l ||\partial^{l+1}_t v||_\infty +\Delta x^l||\partial^{l+1}_x v||_\infty\right).
	\end{align*}
\end{description}
It is interesting to notice that we are not making any assumption on the stability of the high-order scheme, that is because filtered schemes are able to stabilize a possibly unstable scheme.

Before giving some examples of high-order schemes satisfying (A1)-(A3), let us state an interesting property of the solution $v$ of (\ref{HJ}) in case of enough regularity. Notice that we are considering the simplest case of $H$ dependent only on the derivative of $v$.
%%%%%%%%%%%%%%%%%%%%%%%%%%%%%%%%%%%
%VERIFICARE SE SI PUO' ESTENDERE ED IN CASO SCRIVERLO!
%%%%%%%%%%%%%%%%%%%%%%%%%%%%%%%%%%%
\begin{lemma}
	\label{lemma_vtk}
	Let $v$ be the solution of (\ref{HJ}). Then, if $v \in C^{r}\left(\Omega_{(t,x)}\right)$, $r\geq2$, where $\Omega_{(t,x)}$ is a neighborhood of a point $(t,x)\in \Omega:=[0,T]\times \mathbb R$, it holds
	\begin{align}
		\label{vtk_formula}
		\frac{\partial^k v(t,x)}{\partial t^k}&=(-1)^k \frac{\partial^{k-2}}{\partial x^{k-2}}\left(H_p^{k}(v_x(t,x))v_{xx}(t,x)\right)\\
		&=(-1)^k \frac{\partial^{k-2}}{\partial x^{k-2}}\left(H_p^{k-1}(v_x(t,x))\frac{\partial}{\partial x}H(v_{x}(t,x))\right), \nonumber
	\end{align}
	for $k=2,\dots,r$.
\end{lemma}
\begin{proof}
	Let us proceed by induction on $2\leq k \leq r$, omitting the dependence on $(t,x)$ to simplify the notation. For $k=2$, we have
	$$
	v_{tt}=\frac{\partial}{\partial t}(-H(v_x))=-H_p(v_x)v_{xt}=-H_p(v_x)\frac{\partial}{\partial x}(-H(v_x))=H^2_p(v_x)v_{xx},
	$$
	and the statement holds in this case. Suppose now that (\ref{vtk_formula}) holds for $2<k<r-1$, then we can compute
	\begin{align*}
		\frac{\partial^{k+1} v}{\partial t^{k+1}}&=\frac{\partial}{\partial t}\left(\frac{\partial^{k} v}{\partial t^k}\right) \\
		&=\frac{\partial}{\partial t}\left((-1)^k \frac{\partial^{k-2}}{\partial x^{k-2}}\left(H_p^{k}(v_x))v_{xx}\right)\right)\qquad\qquad \textrm{ by inductive hypothesis } \\
		&=(-1)^k \frac{\partial^{k-2}}{\partial x^{k-2}}\left(\frac{\partial}{\partial t}\left(H_p^{k}(v_x))v_{xx}\right)\right) \\
		&=(-1)^k \frac{\partial^{k-2}}{\partial x^{k-2}}\left(\frac{\partial}{\partial p}\left(H^k_p(v_x)\right)v_{xt}v_{xx}+H^k_p(v_x)v_{xxt}\right)\\
		&=(-1)^k \frac{\partial^{k-2}}{\partial x^{k-2}}\left(\frac{\partial}{\partial x}(H^k_p(v_x))v_{xt}+H^k_p(v_x)\frac{\partial }{\partial x}(v_{xt})\right)\\
		&=(-1)^k \frac{\partial^{k-1}}{\partial x^{k-1}}\left(H^k_p(v_x)v_{tx}\right)\\
		&=(-1)^{k+1}\frac{\partial^{k-1}}{\partial x^{k-1}}\left(H^{k+1}_p(v_x)v_{xx}\right),
	\end{align*}
	as we wanted.  
%	\qed
\end{proof}
Let us now consider the value of the solution at $v(t+\Delta t,x)$, with $\Delta t>0$ and its Taylor expansion of order $r \geq 2$ around the point $(t,x)$. Using Lemma \ref{lemma_vtk}, we can rewrite
\begin{align}
	\label{exp_v_lw}
	v(t+\Delta t,x)&=v(t,x)+\Delta t v_t(t,x)+\sum_{k=2}^{r}\frac{\Delta t^k}{k!}\frac{\partial^k v(t,x)}{\partial t^k}+O(\Delta t^{r+1}) \nonumber \\
	&=v(t,x)-\Delta t H(v_x(t,x)) +\nonumber \\
	&\qquad\sum_{k=2}^{r}\frac{(-\Delta t)^k}{k!} \frac{\partial^{k-2}}{\partial x^{k-2}}\left(H_p^{k}(v_x(t,x))v_{xx}(t,x)\right)+O(\Delta t^{r+1}),
\end{align}
which for $r=2$ simply reads
\begin{equation}
	\label{lax_wen}
	v(t+\Delta t,x)=v(t,x)-\Delta t H(v_x(t,x))+\frac{\Delta t^2}{2}H^2_p(v_x(t,x))v_{xx}(t,x)+O(\Delta t^3).
\end{equation}
\begin{remark}
	Using this last relation we could show that, assuming (A1)-(A2), the consistency property is equivalent to require that for $l=2,\dots,k$, and for all $v \in C^{l+1},$
	\begin{align}
		\mathcal E_{A}(v)(t,x):=&\left|h^A(D^-v,D^+v)-H(v_x)+\frac{\Delta t}{2}H^2_p(v_x)v_{xx}\right|\nonumber\\
		&\leq C_{A,l}\left(\Delta t^l ||\partial^{l+1}_t v||_\infty +\Delta x^l||\partial^{l+1}_x v||_\infty\right).
		\label{cond_nec_ha}
	\end{align}
\end{remark}
Now, let us give some examples of high-order schemes satisfying (A1)-(A3) with $l=2$.
\begin{example}
	As a first example let us consider the class of schemes obtained combining a \emph{high-order in space numerical hamiltonian} $h^A_*$ and the second order \emph{Runge-Kutta SSP} (or \emph{Heun scheme}). To explain the simple procedure, let us consider the semidiscrete problem
	$$
	u_t=h_*^A(D^-u(t,x),D^+u(t,x))),
	$$
	where $h_*^A$, is a high-order in space numerical hamiltonian of second order,
	\begin{equation}
		h_*^A(D^-v^n_j,D^+v^n_j)=H(v_x(t^n,x_j))+O(\Delta x^2),
		\label{cond_ha_space}
	\end{equation}
	such as the simple second order \emph{central approximation}
	\begin{equation}
		\label{heun}
		h_*^A(D^-u^n_j,D^+u^n_j)=H\left(\frac{D^-u^n_j+D^+u^n_j}{2}\right),
	\end{equation}
	then to obtain the same accuracy in time we discretize using the second order SSP Runge-Kutta scheme, 
\begin{equation}
		\label{RK2}
		\left\{
		\begin{array}{l}
			u^*_j=u^n_j-\Delta t h_*^A(D^{-}u^n_j,D^{+}u^n_j)\\
			u^{n+1}_j=\frac{1}{2}u^n_j+\frac{1}{2}u^*_j-\frac{\Delta t}{2}h_*^A(D^{-}u^*_j,D^{+}u^*_j).
		\end{array}
		\right.
\end{equation}
The scheme can be written in differenced form in the sense of (A1)-(\ref{cond_nec_ha}) defining
	\begin{equation}
		h^A(D^{-}u^n_j,D^{+}u^n_j)=\frac{1}{2}\left[h_*^A(D^{-}u^n_j,D^{+}u^n_j)+h_*^A(D^{-}u^*_j,D^{+}u^*_j)\right].
	\end{equation}
Notice that through this procedure the stencil of the scheme (\ref{heun}) becomes doubled for $h^A$. 
Notice also that this procedure can be easily extended to the case of hamiltonian dependent on the space variable $x$.
\end{example}
\begin{example}
	Then we propose a couple of numerical hamiltonians $h^A$ obtained discretizing directly the formula (\ref{lax_wen}) or, equivalently, obtained from the same \emph{Lax-Wendroff schemes} for conservation laws by the substitution $u^n_j=\frac{v^n_{j+1}-v^n_j}{\Delta x}$. The first is the original \emph{Lax-Wendroff} scheme
	\begin{equation}
		\begin{array}{ll}
			h^A(D^{-}u_j^n,D^{+}u_j^n)=&\frac{1}{2} \left\{H\left(D^+ u_j^n\right)+H\left(D^-u^n_j\right)+\right.\\
			&\quad\left. -\frac{\Delta t}{\Delta x}H_p\left(\frac{D^-u^n_j+D^+u^n_j}{2}\right)\left[H\left(D^+u^n_j\right)-H\left(D^-u^n_j\right)\right]\right\},
		\end{array}
		\label{lw}
	\end{equation}
	and the second is its variation proposed by \emph{Richtmyer},
	\begin{equation}
		\label{lwr}
		h^A(D^{-}u_j^n,D^{+}u_j^n)= H\left(\frac{D^-u^n_j+D^+u^n_j}{2}-\frac{\Delta t}{2 \Delta x}\left[H\left(D^+u^n_j\right)-H\left(D^-u^n_j\right)\right]\right).
	\end{equation}
\end{example}
\begin{example}
	Following the approach of the Lax-Wendroff schemes and making use of the expansion (\ref{exp_v_lw}), we can easily write higher order schemes, in both space and time, using very compact stencils. The idea is simply to discretize directly the above expansion using finite difference approximations of the right order. For example, if we want to write a \emph{fourth order Lax-Wendroff scheme} using only five points, one of the possibilities is to define
	$$
	\begin{array}{ll}
	H_1&=H\left(\frac{u_{j-2}-8u_{j-1}+8u_{j+1}-u_{j+2}}{12\Delta x}\right),\\
	H_2&=H^2_p\left(\frac{u_{j-2}-8u_{j-1}+8u_{j+1}-u_{j+2}}{12\Delta x}\right)\left(\frac{-u_{j-2}+16u_{j-1}-30u_j+16u_{j+1}-u_{j+2}}{12\Delta x^2}\right),\\
	H_3&=\frac{1}{2\Delta x}\left[H_p^3\left(\frac{u_{j+2}-u_j}{2\Delta x}\right)\left(\frac{u_{j+2}-2u_{j+1}+u_j}{\Delta x^2}\right)-H_p^3\left(\frac{u_{j}-u_{j-2}}{2\Delta x}\right)\left(\frac{u_{j}-2u_{j-1}+u_{j-2}}{\Delta x^2}\right)\right],\\
	H_4&=\frac{1}{\Delta x^2}\left[H_p^4\left(\frac{u_{j+2}-u_j}{2\Delta x}\right)\left(\frac{u_{j+2}-2u_{j+1}+u_j}{\Delta x^2}\right)-2H_p^4\left(\frac{u_{j+1}-u_{j-1}}{2\Delta x}\right)\left(\frac{u_{j+1}-2u_{j}+u_{j-1}}{\Delta x^2}\right)\right.\\
	&\left.\qquad+H_p^4\left(\frac{u_{j}-u_{j-2}}{2\Delta x}\right)\left(\frac{u_{j}-2u_{j-1}+u_{j-2}}{\Delta x^2}\right)\right],
	\end{array}
	$$
	and then compute
	\begin{equation}
		\label{lw_4ord}
		h^A(D^-u_j^n,D^+u_j^n)=H_1-\frac{\Delta t}{2}\left[H_2-\frac{\Delta t}{3}\left(H_3-\frac{\Delta t}{4}H_4\right)\right].
	\end{equation}
	It is straightforward to verify that, if the solution $v$ is regular enough, using Taylor expansion we have
	\begin{itemize}
		\item $H_1=H(v_x)+O(\Delta x^4)$,
		\item $H_2=H^2_p(v_x)v_{xx} +O(\Delta x^4)$,
		\item $H_3=\frac{\partial}{\partial x}\left(H_p^3(v_x)v_{xx}\right)+O(\Delta x^2)$,
		\item $H_4=\frac{\partial^2}{\partial x^2}\left(H_p^4(v_x)v_{xx}\right)+O(\Delta x^2)$,
	\end{itemize}
	and that the resulting scheme satisfies (A1)-(A3) with $l=4$. Notice that to obtain fourth order it would have been enough to have approximations of one order lower for $H_2$ and $H_4$, but thanks to the symmetry of the discretizations we can get higher orders without increasing the number of points in the stencil.
\end{example}
%%%%%%%%%%%%%%%%%%%%%%%%%%%%%%%%%%%%%%%%%%%%%%%
%%%%%%%%%%%%%%%%%%%%%%%%%%%%%%%%%%%%%%%%%%%%%%%
\subsection{Filter function}\label{subsec:filters}
In order to couple the schemes and their properties, we need to define a function $F$, called \emph{filter function F}, such that
\begin{description}
	%\label{prop_F}
	\item [(F1)] $F(x) \approx x$ for $|x|\leq 1$,
	\item [(F2)] $F(x) =0$ for $|x|>1$,
\end{description}
which implies that
\begin{itemize}
	\item If $| S^A-S^M|\leq \Delta t \varepsilon^n$ and $\phi_j^n=1\Rightarrow S^{AF}\approx S^A$
	\item If $| S^A-S^M|> \Delta t \varepsilon^n$ or $\phi^n_j=0 \Rightarrow S^{AF}= S^M$.
\end{itemize}
It is clear that, with just these two requirements, several filter functions can be considered, which differ for regularity properties. Four examples are reported in Fig. \ref{fig_filtri}. 
	The first filter, $F_1$, which we use in our numerical tests, has been defined in 
	\cite{BFS16} as
	\begin{equation}\label{def_F1}
		F_1(x)=\left\{
		\begin{array}{ll}
			x\qquad&\textrm{ if } |x|\leq 1 \\
			0&\textrm{ otherwise,}
		\end{array}
		\right.
	\end{equation}
	which is clearly discontinuous at $x=-1,1$. \\
As a second possibility, we propose the family of regular filter functions given by the formula
	$$
	F(x)=x\exp\left(-c(|x|-a)^b\right),
	$$
	for appropriate choices of the parameters $a$, $b$ and $c$. 
In Fig. \ref{fig_filtri}, the filter $F_2$ belongs to that family with $a=0.25,\ b=20,\ c=4$. 
Functions of that kind are very regular ($F\in C^\infty$) and developing with Taylor we can see that they satisfy the properties (F1)-(F2). \\
Another example of filter functions satisfying (F1)-(F2), continuous but not necessarily derivable, is the following family of functions
\begin{equation}
F(x)=\left\{
\begin{array}{ll}
x \exp\left(-\frac{a}{b-|x|}\right)\qquad&\textrm{ if } |x|\leq b \\
0 &\textrm{ otherwise},
\end{array}
\right.
\end{equation}
varying the parameters $a$ and $b$. 
In Fig. \ref{fig_filtri} we show $F_3(x)$ with $a = 0.001$ and $b=1.05$, for which the function approach better the value $1$ at $x=-1,1$.  \\
Finally, we recall also the filter defined in \cite{FO13} as 
	\begin{equation}
		F_4(x)=\left\{
		\begin{array}{ll}
			x\qquad\qquad&|x|\leq 1\\
			0 & |x|\geq 2 \\
			-x+2 &1\leq x \leq 2\\
			-x-2 &-2\leq x \leq -1,
		\end{array}
		\right.
	\end{equation}
which satisfies (F1)-(F2) with a rather wide transition phase. 

\begin{figure}[t]
	\begin{center}
		\includegraphics[scale=0.33]{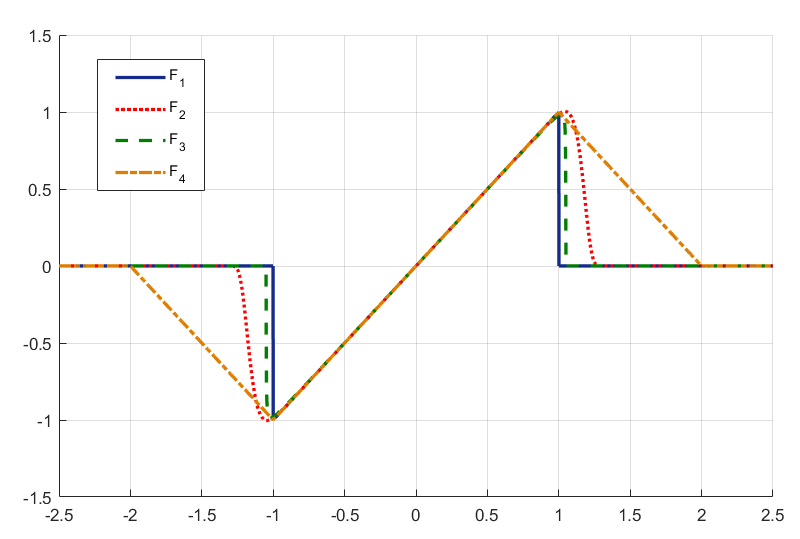}
	\end{center}
	\caption{\footnotesize{Possible choices for the filter function $F$.}\label{fig_filtri}}
\end{figure}
After extensive computations, we noticed that the results obtained with our adaptive filtered (AF) scheme are not sensitive with respect to changes in regularity of the filter function, even with very large transition phases. That is probably because, as we will see in the next section, the parameter $\ep^n$ is designed to obtain the property (F1) whenever possible, then in regions of regularity of the solution the argument of $F$ lies most probably in $[-1,1]$, where all the filter functions are practically the same. Some major differences, instead, can be seen in the results obtained with the basic filtered %SF 
scheme, for which the threshold $\ep$ is fixed at the beginning, as it is highlighted in the introduction of \cite{BFS16}.
%%%%%%%%%%%%%%%%%%%%%%%%%%%%%%%%%%%%%%%%%%%%%%%%%%%%%%%%%%%%%%%%%%%
%%%%%%%%%%%%%%%%%%%%%%%%%%%%%%%%%%%%%%%%%%%%%%%%%%%%%%%%%%%%%%%%%%
\subsection{Tuning of the parameter $\varepsilon^n$}
\label{sect_tune_eps}
The last step is to show how to compute the switching parameter $\varepsilon^n$, which is the real core of the adaptivity of our scheme. Then, if we want the scheme (\ref{AFS}) to switch to the high-order scheme when some regularity is detected, we have to choose $\varepsilon^n$ such that
\begin{equation}
	\label{ep_ineq}
	\left|\frac{S^A(v^n)_j-S^M(v^n)_j}{\varepsilon^n \Delta t}\right|=\left|\frac{h^A(\cdot)-h^M(\cdot)}{\varepsilon^n}\right| \leq 1,\qquad \textrm{ for }(\Delta t,\Delta x)\to 0,
\end{equation}
in the region of regularity at time $t_n$, that is 
\begin{equation}\label{def_region_Rn}
	\mathcal R^n=\left\{x_j : \phi^n_j=1\right\}.
\end{equation} For the moment, to simplify the presentation we assume the existence of a function $\phi$ such that
\begin{equation}
	\label{phi}
	\phi^n_j=\left\{
	\begin{array}{ll}
		1\qquad&\textrm{ if the solution } u^n\textrm{ is regular in }I_j,  \\
		0&\textrm{ if }I_j \textrm{ contains a point of singularity,}
	\end{array}
	\right.
\end{equation}
where $I_j=(x_{j-1},x_{j+1})$, 
referring to the next section for some examples of practical computation of the function $\phi$.

Assuming $v$ sufficiently regular and proceeding by Taylor expansions as in \cite{BFS16}, we have for the monotone scheme 
$$
h^M(D^-v^n_j,D^+ v^n_j)=H(v^n_x(x_j))+\frac{\Delta x}{2}v^n_{xx}(x_j)\left(\partial_{p_+}h^M_j-\partial_{p_-}h^M_j\right)+O(\Delta x^2),
$$
where we used the relation
$$
D^\pm v^n_j=v^n_x(x_j))\pm\frac{\Delta x}{2}v^n_{xx}(x_j)+O(\Delta x^2),
$$
whereas for the high-order scheme, by the consistency property,
$$
h^A(D^-v^n_j,D^+ v^n_j)=H(v^n_x(x_j))-\frac{\Delta t}{2}H^2_p(v^n_x)v_{xx} +O(\Delta t^2)+O(\Delta x^2).
$$
Whence, from (\ref{ep_ineq}) we obtain
\begin{equation}
	\label{epn_ineq}
	\ep^n\geq \left|\frac{\Delta x}{2}v^n_{xx}\left(\partial_{p_+}h^M_j-\partial_{p_-}h^M_j+\lambda H^2_p(v^n_x)\right)+O(\Delta t^2)+O(\Delta x^2)\right|.
\end{equation}
Finally, we use a numerical approximation of the lower bound on the right hand side of the previous inequality to obtain the following formula for $\varepsilon^n$,
\begin{align}
	\label{eps}
	\varepsilon^n=&\max_{x_j\in\mathcal R^n}K\left|H\left(D \, u^n_j\right)-H\left(D \, u^n_j - \lambda\left[H(D^+ u^n_j)-H(D^-u^n_j)\right]\right)\right.\nonumber \\
	&\quad+\left[h^M(D \, u^n_j,D^+ u^n_j)-h^M(D \, u^n_j,D^-u^n_j)\right] \nonumber\\
	&\quad\quad\left.-\left[h^M(D^+ u^n_j,D \, u^n_j)-h^M(D^-u^n_j,D \, u^n_j)\right]\right|,
\end{align}
with $K>\frac{1}{2}$, $\lambda := \frac{\Delta t}{\Delta x}$ and $D \, u^n_j := \frac{u^n_{j+1}-u^n_{j-1}}{2\Delta x}$. Notice that if we assume enough regularity on the solution $v$, then (\ref{eps}) gives a second order approximation of the right hand side of (\ref{epn_ineq}) multiplied by $2K$.

%%%%%%%%%%%%%%%%%%%%%%%%%%%%%%%%%%%%%%%%%%%%%%
%%%%%%%%%%%%%%%  SECTION 3: smoothness indicators
%%%%%%%%%%%%%%%%%%%%%%%%%%%%%%%%%%%%%%%%%%%%%%
\section{Smoothness indicator function}
\label{sec:indicators}
In the previous section we assumed the existence of a \emph{smoothness indicator function} $\phi$, in the sense that
\begin{equation}\label{def_phi}
	\phi^n_j=\phi(\omega^n_j):=\left\{
	\begin{array}{ll}
		1\qquad&\textrm{ if the solution } u^n\textrm{ is regular in }I_j,  \\
		0&\textrm{ if }I_j \textrm{ contains a point of singularity,}
	\end{array}
	\right.
\end{equation}
where $I_j=(x_{j-1},x_{j+1})$ and $\omega^n_j$ is the \emph{smoothness indicator} at the node $x_j$ depending on the values of the approximate solution $u^n$. The aim of this section is precisely to show a simple construction of a function satisfying (\ref{phi}) which makes use of smoothness indicators widely known in literature. Moreover, in the process we review the theory of the smoothness indicators of \cite{JP00}, defined for the construction of the WENO schemes for (\ref{HJ}),
\begin{equation}
	\label{beta}
	\beta_k = \beta_k(u^n)_j := \sum_{l=2}^r\int_{x_{j-1}}^{x_{j}} \Delta x^{2l-3} \left( P_k^{(l)}(x)\right)^2 dx,
\end{equation}
for $k=0,\dots,r-1$, where $P_k$ is the Lagrange polynomial of degree $r$ interpolating the values of $u^n$ on the stencil $\mathcal S_{j+k}=\{x_{j+k-r},\dots,x_{j+k}\}$.

Then, before proceeding with the construction of $\phi$, let us state a fundamental result on the behavior of the indicators (\ref{beta}).
\begin{proposition}
	\label{prop_beta}
	Assume $f\in C^{r+1}\left(\Omega\setminus\{x_s\}\right)$, with $\Omega$ a neighborhood of $x_s$, and $f^\prime(x_s^-)\not = f^\prime(x_s^+)$. 
	Moreover, just for simplicity, let $f''(x) \neq 0$, $\forall x \in \Omega \setminus {x_s}$. 
	Then, for $k=0,\dots,r-1$ and $j\in \mathbb Z$, the followings are true:
	\begin{enumerate}[i)]
		\item If $x_s \in \Omega\ \setminus \stackrel{\circ}{\mathcal S}_{j+k}\quad \Rightarrow \beta_{k}(f)=O(\Delta x^2)$,
		\item If $x_s \in \ \stackrel{\circ}{\mathcal S}_{j+k}\quad \Rightarrow \beta_k(f)=O(1)$,
	\end{enumerate}
	where 
	$\stackrel{\circ}{\mathcal S}_{j+k}=(x_{j-r+k},x_{j+k})$. 
\end{proposition}
We skip the proof, which is rather technical, but the interested reader can find it in Appendix \ref{sec:appendix}.
\begin{remark}
	\label{def_sigma}
	Notice that we could avoid the restrictions on $f^{\prime\prime}$ at points of regularity by adding a small quantity $\sigma_{\Delta x} := \sigma \Delta x^2$, for some constant $\sigma>0$, to the indicators $\beta_k$ and consider instead
	\begin{equation}\label{beta_k_sigma}
		\widetilde{\beta}_k:=\beta_k+\sigma_{\Delta x}. 
	\end{equation}
This is necessary in order to avoid a reduction of accuracy at points such that $f^{\prime\prime} =0$, as it has been thoroughly discussed in \cite{ABM10} in the case of discontinuous functions. We will use this assumption in our numerical tests, choosing $\sigma=1$.
\end{remark}
Our aim is to identify the points (or the intervals) in which a function $f$ presents a singularity in the first derivative using only its nodal values $f_j$, $j\in \mathbb Z$.  
Let us focus the attention on a point $x_j$ of the grid and consider the simplest case of $r=2$, which is enough for our purpose.
Let us consider separately the intervals $(x_{j-1},x_{j}]$ and $[x_{j},x_{j+1})$ defining
\begin{equation}\label{beta_k_m}
	\beta^-_k := \Delta x\int_{x_{j-1}}^{x_{j}} ( P_k^{\prime\prime}(x))^2 dx  \,\, =\left(\frac{f_{j+k}-2f_{j+k-1}+f_{j+k-2}}{\Delta x}\right)^2,
\end{equation}
for $k=0,1$, where $P_0$, $P_1$ are the polynomials interpolating the function, respectively, on the stencils $\{x_{j-2},x_{j-1},x_{j}\}$ and $\{x_{j-1},x_{j},x_{j+1}\}$, and symmetrically
\begin{equation}\label{beta_k_p}
	\beta^+_k := \Delta x\int_{x_{j}}^{x_{j+1}} ( P_k^{\prime\prime}(x))^2 dx \,\, = \left(\frac{f_{j+k+1}-2f_{j+k}+f_{j+k-1}}{\Delta x}\right)^2,
\end{equation}
for $k=0,1$, where now $P_0$, $P_1$ are the interpolating polynomials on the stencils $\{x_{j-1},x_{j},x_{j+1}\}$ and $\{x_{j},x_{j+1},x_{j+2}\}$. From the definition it is clear that $(\beta^+)_j=(\beta^-)_{j+1}$ so we have to compute the quantities just once (note that these quantities are always nonnegative). 
Then, we define as in \cite{JP00}
\begin{equation}
	\label{def_alpha_pm}
	\alpha^\pm_k=\frac{1}{(\beta^\pm_k+\sigma_{\Delta x})^2},
\end{equation}
with $\sigma_{\Delta x} := \sigma \Delta x^2$ the parameter we introduced in Remark \ref{def_sigma}, and focus on the information given by the interpolating polynomial on $\{x_{j-1},x_{j},x_{j+1}\}$ defining 
$$
\omega_+=\frac{\alpha^+_0}{\alpha^+_0+\alpha^+_1} \quad \textrm{ and }\quad\omega_-=\frac{\alpha^-_1}{\alpha^-_0+\alpha^-_1},
$$
to inspect the regularity, respectively, on $[x_j,x_{j+1})$ and for $(x_{j-1},x_j]$. 

By Prop. \ref{prop_beta} and Remark \ref{def_sigma} we know that $\widetilde\beta_k= O(\Delta x^2)$ if there is no singularity in the stencil, and $\widetilde\beta_k=O(1)$ otherwise, so in presence of a singularity we can only fall in one of the following cases:
\begin{itemize}
	\item If $x_{j-2}<x_s\leq x_{j-1}$, then ${\widetilde\beta}^-_0=O(1)$, ${{\widetilde\beta}}^-_1={\widetilde\beta}^+_0=O(\Delta x^2)$, ${\widetilde\beta}^+_1=O(\Delta x^2)$,
	\item If $x_{j-1}<x_s< x_{j}$, then ${\widetilde\beta}^-_0=O(1)$, ${\widetilde\beta}^-_1={\widetilde\beta}^+_0=O(1)$, ${\widetilde\beta}^+_1=O(\Delta x^2)$,
	\item If $x_s= x_{j}$, then ${\widetilde\beta}^-_0=O(\Delta x^2)$, ${\widetilde\beta}^-_1={\widetilde\beta}^+_0=O(1)$, ${\widetilde\beta}^+_1=O(\Delta x^2)$,
	\item If $x_{j}< x_s< x_{j+1}$, then ${\widetilde\beta}^-_0=O(\Delta x^2)$, ${\widetilde\beta}^-_1={\widetilde\beta}^+_0=O(1)$, ${\widetilde\beta}^+_1=O(1)$,
	\item If $x_{j+1}\leq x_s< x_{j+2}$, then ${\widetilde\beta}^-_0=O(\Delta x^2)$, ${\widetilde\beta}^-_1={\widetilde\beta}^+_0=O(\Delta x^2)$, ${\widetilde\beta}^+_1=O(1)$,
\end{itemize}
with $x_s$ point of singularity. Now, we can compute 
\begin{eqnarray}
	\frac{\alpha_1^\pm-\alpha_0^\pm}{\alpha_0^\pm}&=&\frac{(\beta_0^\pm+\sigma_{\Delta x})^2-(\beta_1^\pm+\sigma_{\Delta x})^2}{(\beta_1^\pm+\sigma_{\Delta x})^2} \nonumber \\
	&=&\left(\frac{\beta_0^\pm-\beta_1^\pm}{\beta_1^\pm+\sigma_{\Delta x}}\right)\left(\frac{\beta^\pm_0+\beta^\pm_1+2\sigma_{\Delta x}}{\beta_1^\pm+\sigma_{\Delta x}}\right),
	\label{expansion_alpha}
\end{eqnarray}
which, noticing that, if the function is smooth in both stencils of $\beta_0^\pm$ and $\beta_1^\pm$, we have
\begin{align}
	\label{sviluppo_tau}
	&\frac{\beta_0^\pm-\beta_1^\pm}{\beta_1^\pm+\sigma_{\Delta x}}=-2\Delta x\frac{f_j^{\prime \prime}f_j^{\prime\prime\prime}}{(f^{\prime\prime})^2+\sigma}+O(\Delta x^2)= O(\Delta x)\\
	&\frac{\beta^\pm_0+\beta^\pm_1+2\sigma_{\Delta x}}{\beta_1^\pm+\sigma_{\Delta x}}=2+O(\Delta x)=O(1), \nonumber
\end{align}
leads to
\begin{equation}
	\label{relazione_alpha}
	\alpha^\pm_1=\alpha^\pm_0(1+O(\Delta x)).
\end{equation}
Whence we can deduce that if the solution is regular enough in both stencils 
\begin{equation}
	\label{relazione_omega}
	\omega_\pm = \frac{1}{2}+O(\Delta x).
\end{equation}
On the other hand, if there is a singularity in at least one of the stencils, by Prop. \ref{prop_beta} and the definition \eqref{def_alpha_pm} we have that 
\begin{equation}
	\alpha^\pm_k=
	\left\{
	\begin{array}{ll}
		O(1)\qquad&\textrm{ if } f \textrm{ is not smooth in }\stackrel{\circ}{\mathcal S}_{j+k} \\
		O(\Delta x^{-4})&\textrm{ if } f \textrm{ is smooth in }\stackrel{\circ}{\mathcal S}_{j+k},
	\end{array}
	\right.
	\label{alpha_01}
\end{equation}
then it is easy to verify that the behavior of our $\omega_\pm$ falls in the following cases:
\begin{itemize}
	\item If $x_{j-2}<x_s\leq x_{j-1}$, then $\omega_-= 1+O(\Delta x^4)$, $\omega_+ = 1/2+O(\Delta x)$
	\item If $x_{j-1}<x_s< x_{j}$, then $\omega_-= O(1)$, $\omega_+= O(\Delta x^4)$
	\item If $x_s= x_{j}$, then $\omega_-= O(\Delta x^4)$, $\omega_+= O(\Delta x^4)$
	\item If $x_{j}<x_s< x_{j+1}$, then $\omega_-=O(\Delta x^4)$, $\omega_+ = O(1)$
	\item If $x_{j+1}\leq x_s< x_{j+2}$, then $\omega_-= 1/2+O(\Delta x)$, $\omega_+= 1+O(\Delta x^4)$,
\end{itemize}
where with $\omega_\pm = O(1)$ we mean a number dependent on the jump of the derivative. 
Now, defining $\omega_j := \min\{\omega_-,\omega_+\}$ we can rewrite
\begin{equation}\label{def:omega_j}
	\omega_j=
	\left\{
	\begin{array}{ll}
	O(\Delta x^4)\qquad &\textrm{ if }x_{j-1}<x_s<x_{j+1} \\
	\frac{1}{2}+O(\Delta x)& \textrm{ otherwise.}
	\end{array}
	\right.
\end{equation}
Finally, what is left is to define the function $\phi$ such that $\phi=1$ if $\omega$ is close to $\frac{1}{2}$ and $\phi=0$, otherwise. Notice that in the latter are included both cases in which the function has a singularity in the first derivative ($\omega=O(\Delta x^4)$) and when the second derivative is discontinuous ($\omega=O(1)$). The simplest choice is to take
\begin{equation}
	\label{phi_disc}
	\phi(\omega)=\chi_{\{\omega\geq M\}},
\end{equation}
with $M<\frac{1}{2}$, a number possibly dependent on $\Delta x$.

\begin{remark}
	Notice that to construct the function $\phi$ using the indicators (\ref{beta}) with $r=2$ we need only five points to inspect the regularity in $I_j$.
\end{remark}

Next, we show that if we make a particular choice for $M$ we are able to prove the following result, which can be seen as an ``inverse'' of Prop.  \ref{prop_beta} for numerical solutions and, if we use the previous simple construction for $\omega_j$, gives a useful tool for the analysis of the next section. 

Before proceeding, let us remind that we are working with structured grids, then if we consider a one-parameter family of grid values $\{f_j(\Delta x)\}_{j\in J(\Delta x)}$, as $\Delta x$ goes to 0, the indexed family of sets of indices $J(\Delta x)$ is \emph{expanding}, in the sense that if $\Delta x_2<\Delta x_1$, then $J(\Delta x_1)\subset J(\Delta x_2)$, where $J(\Delta x)\subseteq \mathbb Z$, for all $\Delta x>0$. Moreover, we define $I_s(\Delta x)$ as the set of indices $j$ such that $\phi_j=0$ and assume, for simplicity, $|I_s(\Delta x)|< \overline I_s$, where $\overline I_s$ is a positive constant. 

%%%%%%%%%%%%%% LEMMA 2 
\begin{lemma}\label{lemma_bound_B}
Let $\omega$ be computed using (\ref{beta_k_m})-(\ref{beta_k_p}) and $\phi$ be defined by (\ref{phi_disc}) with $M(\Delta x)=\frac{1}{2}-C\Delta x$, for some constant $C$ such that $0<M(\Delta x)<\frac{1}{2}$. Consider a one-parameter family of sequences $\{f_j(\Delta x)\}_{j\in J(\Delta x)}$ with compact support  in the interval $[-b,b]$, and a partition $\{R_i\}_{i=0,\dots, |I_s|}$ of the regularity set $\mathcal R=\{j\in \mathbb Z: \phi_j=1\}=\bigcup_{i}R_i$, and $\mathcal R=\mathbb Z$ if $I_s=\emptyset$. Then, if for all $i=0,\dots,|I_s|$, there exists $j_{i}\in R_i$, such that $|D^2 f_{j_i}(\Delta x)|<\infty$, we have that
	\begin{equation}
		\label{bound_d2u}
		|D^2 f_j (\Delta x)|=\frac{|f_{j+1}(\Delta x)-2f_j(\Delta x)+f_{j-1}(\Delta x)|}{\Delta x^2}\leq B, \quad \forall j\in \mathcal R,
	\end{equation}
	for a constant $B$ independent of $\Delta x$.
\end{lemma}
\begin{proof}
Since $\{f_j\}$ has compact support we have $|I_s|<\infty$ and it will be  enough to prove the assertion just for one $i\in I_s$. More simply,  in the regular case we have $\mathcal R=\mathbb Z$ and we want  to show  that the statement is true if the discrete derivative is bounded at some point, so if there exist  an index $\widehat j\in \mathcal R$ and a positive constant $\delta$ (independent of  $\Delta x$) such that $|D^2f_{\widehat j}(\Delta x)|\le \delta $  (note that for the nodes outside the support of $f_j$ we can even set $\delta=0$). In the following, we simplify the notation dropping  the dependence of $f_j$ on $\Delta x$.	
By definition of $\phi$ and $\omega$, if $\phi_j=1$ then both $\omega_\pm > M$. Moreover,  (\ref{beta_k_m})-(\ref{beta_k_p}) imply that the coefficients $\beta^{\pm}$ are always  nonnegative as well as $\omega_\pm$. \\
Let us consider the case $j<\widehat j$. 
Then, by definition,
	$$
	\omega_+=\frac{(\beta_1^++\sigma_{\Delta x})^2}{(\beta_1^++\sigma_{\Delta x})^2+(\beta_0^++\sigma_{\Delta x})^2}>M,
	$$
	which leads by simple computations to
	$$
	\beta^+_0<\sqrt{\frac{1-M}{M}}\beta_1^++\left(\sqrt{\frac{1-M}{M}}-1\right)\sigma_{\Delta x},
	$$
	then, dividing by $\Delta x^2$ and recalling that $\sigma_{\Delta x}=\sigma \Delta x^2$ , we get
	\begin{equation} \label{eq:iter}
	|D^2 f_j|^2<\sqrt{\frac{1-M}{M}}|D^2 f_{j+1}|^2+\left(\sqrt{\frac{1-M}{M}}-1\right)\sigma. 
	\end{equation}
	Now let us iterate \eqref{eq:iter} on $j$ till $\widehat j$ and define $L_j\equiv \widehat j-j$,   we have
	\begin{align*}
		|D^2 f_j|^2<\dots&< \left(\frac{1-M}{M}\right)^{\frac{L_j}{2}}|D^2 f_{\widehat j}|^2+\left(\sqrt{\frac{1-M}{M}}-1\right) \sigma \sum_{k=0}^{L_j-1}\left(\frac{1-M}{M}\right)^{\frac{k}{2}}\\
		&< \left(\frac{1-M}{M}\right)^{\frac{L_j}{2}} \delta^2 +\left(\sqrt{\frac{1-M}{M}}-1\right)\sigma \sum_{k=0}^{L_j-1}\left(\frac{1-M}{M}\right)^{\frac{k}{2}}\\
		&= \left(\frac{1-M}{M}\right)^{\frac{L_j}{2}} \delta^2 +\left(\sqrt{\frac{1-M}{M}}-1\right)\sigma \frac{1-\left(\frac{1-M}{M}\right)^{\frac{L_j}{2}}}{1-\sqrt{\frac{1-M}{M}}}\\
		&= \left(\frac{1-M}{M}\right)^{\frac{L_j}{2}}(\sigma+ \delta^2) -\sigma.
	\end{align*}
For  $j>\widehat j$, we can use the relation $\omega_->M$ and iterate back to $\widehat j$ redefining  $L_j \equiv  j-\widehat j$, the calculations are similar also for this case.\\
Since $L_j \Delta x$ is bounded by $b$,  we have $L_j\leq \frac{b}{\Delta x}$, $\forall j\in \mathcal R$. Recalling that  $M=\frac{1}{2}-C\Delta x$, we can use the previous bound on $|D^2 f_j|^2$ to proceed
	\begin{align}
	\label{eq:bound}
		|D^2f_j|^2&\leq \left(\frac{1}{M} -1 \right)^{\frac{b}{2\Delta x}} (\sigma+\delta^2)-\sigma  = \left(\frac{2}{1-2C\Delta x}-1\right)^{\frac{b}{2\Delta x}}(\sigma+\delta^2)-\sigma 
	\end{align}
We get the final bound  passing to the limit for $\Delta x$ going to 0 in \eqref{eq:bound},  in conclusion we get 
\begin{equation}
|D^2f_j|^2\leq e^{3Cb} (\sigma+\delta^2) -\sigma
\end{equation}
and  the statement follows simply taking $B:=\sqrt{e^{3Cb} (\sigma+\delta^2)-\sigma}$.
%\qed
\end{proof}
%%%%%%%%%%%%%%FINE LEMMA

Unfortunately, we noticed through numerical tests that the $O(\Delta x)$ term in regular regions may produce heavy oscillations around the optimal value $\overline\omega= 1/2$. 
To increase the accuracy, we can use higher order smoothness indicator  ($r>2$), but we would need a bigger reconstruction stencil. Otherwise, if we want to keep the compactness of the stencil, we can use the \emph{mappings} defined in \cite{HAP05}, 
\begin{equation}
	g(\omega)=\frac{\omega(\overline \omega+\overline \omega^2-3 \overline \omega \omega +\omega^2)}{\overline \omega^2 +\omega(1-2\overline \omega)},\qquad \overline \omega \in (0,1),
	\label{map}
\end{equation}
which have the properties that $g(0)=0$, $g(1)=1$, $g(\overline \omega)=\overline \omega$, $g^\prime(\overline \omega)=0$ and $g^{\prime\prime}(\overline \omega)=0$. Then, we define
\begin{align*}
	\omega_\pm^*&=g(\omega_\pm)\\
	&=g(\overline \omega)+g^\prime(\overline \omega)(\omega_\pm-\overline \omega)+\frac{g^{\prime\prime}(\overline \omega)}{2}(\omega_\pm-\overline \omega)^2+\frac{g^{\prime\prime\prime}(\overline \omega)}{6}(\omega_\pm-\overline \omega)^3+O(\Delta x^4)\\
	&=\overline \omega+\frac{(\omega_\pm-\overline \omega)^3}{\overline \omega-\overline \omega^3}+O(\Delta x^4)\\
	&=\overline \omega +O(\Delta x^3). 
\end{align*}
Notice that with respect to the definition in \cite{HAP05} we avoided the second weighting which seems unnecessary in our case. More explicitly, the mapping we use is
	\begin{equation}
		\label{map2}
		g(\omega)=4\omega\left(\frac{3}{4}-\frac{3}{2}\omega+\omega^2\right).
	\end{equation}
It is important to remind that, at the moment, Lemma \ref{lemma_bound_B} is valid only for indicators $\omega$ using the standard construction for $r=2$, without the possibility to introduce any modification, or higher order indicators. Moreover, as it will be briefly discussed in Remark \ref{problema_lemma_bound_B}, it introduces some limitations in the applicability even when using the standard indicators, testifying the necessity of some improvements in the argument used. 
Notice that the previous lemma strongly relies on the fact that $\omega$ is computed using (\ref{beta_k_m})-(\ref{beta_k_p}) without introducing the mappings (\ref{map2}). In fact, if we were to use (\ref{map2}), we could develop the algebra until the inequality

	$$
|D^2f_j|^2\leq \left(\frac{1}{g^{-1}(M)}-1\right)^{\frac{b}{2\Delta x}}	(\sigma+\delta^2)-\sigma,
	$$
but, by definition, $g^{-1}$ cannot be expanded in Taylor series around the point $\frac{1}{2}$, whence we could not use the notable limit to conclude. 

Therefore, we are forced to add a ``technical'' assumption in order to justify the proof of Prop. \ref{prop_eps_n_un}. More precisely, when using the alternative constructions for $\omega$ (using the mapping (\ref{map2})), we define the \emph{region of regularity} $\mathcal R$ detected by the function $\widetilde\phi$ as the set
\begin{equation}\label{def_tilde_phi}
	\mathcal R=\left\{j\in \mathbb Z: \widetilde \phi(\omega_j)=1 \right\},
	\quad \textrm{ with }\quad 
	\widetilde \phi_j=\left\{
	\begin{array}{ll}
		1\quad&\textrm{ if }\phi(\omega_j)=1 \textrm{ and } |Du^2_j|<B,\\
		0&\textrm{ otherwise,}
	\end{array}
	\right.
\end{equation}
for some constant $B\gg0$. Notice that with this definition, which, we recall, is needed only for theoretical reasons, it is not necessary to require $M(\Delta x)\to 0$, then we can simply choose a constant $M>0$ small enough (e.g. $M=0.1$), as we will do in the numerical tests of Sect.  \ref{sec:tests}.

%%%%%%%%%%%%%%%%%%%%%%%%%%%%%%%%%%%%%%%%%%%%%%
%%%%%%%%%%%%%%%  SECTION 4: convergence
%%%%%%%%%%%%%%%%%%%%%%%%%%%%%%%%%%%%%%%%%%%%%%
\section{Convergence result}\label{sec:convergence}
We are now able to present our main result, but before doing so let us state a useful proposition about the numerical solution and the parameter $\varepsilon^n$.
\begin{proposition}\label{prop_eps_n_un}
	Let $u^n$ be the solution obtained by the scheme (\ref{AFS})-(\ref{eps}) and assume that $v_0$ and $H$ are Lipschitz continuous functions. Assume also that $\mathcal R^n$ is defined by (\ref{def_region_Rn}) or (\ref{def_tilde_phi}), with $\phi$ given by (\ref{phi_disc}), and that $\lambda=\Delta t/\Delta x$ is a constant such that \eqref{cond_CFL} is satisfied. Then, $\varepsilon^n$ is well defined and $u^n$ satisfies, for any $i$ and $j$, the discrete Lipschitz estimate
	\begin{equation}
		\label{lip_un}
		\frac{|u^n_i-u^n_j|}{\Delta x}\leq L
	\end{equation}
	for some constant $L>0$, for $0\leq n\leq T/\Delta t$. Moreover, there exists a constant $C>0$ such that 
	\begin{equation}
		\label{bound_eps_n}
		\varepsilon^n\leq C \Delta x.
	\end{equation}
\end{proposition}
\begin{proof}
	Before proceeding with the proof let us notice that, if $u^n$ satisfies (\ref{lip_un}) for a constant $L_n>0$, calling for brevity
	$$
	D^*u_j:=D \, u^n_j-\lambda\left[H(D^+ u^n_j)-H(D^-u^n_j)\right],
	$$
	we have that
	\begin{align*}
		\varepsilon^n&=\max_{x_j\in\mathcal R^n}K\left|H\left(D \, u^n_j\right)-H\left(D^*u_j\right)+\left[h^M(D \, u^n_j,D^+ u^n_j)-h^M(D \, u^n_j,D^-u^n_j)\right] \nonumber\right.\\
		&\quad\quad\left.-\left[h^M(D^+ u^n_j,D \, u^n_j)-h^M(D^-u^n_j,D \, u^n_j)\right]\right|\nonumber \\
		&=\max_{x_j\in\mathcal R^n}K\left|\left[\Delta t\left(\frac{H\left(D \, u^n_j\right)-H\left(D^*u_j\right)}{D \, u^n_j-D^*u_j}\right)\left(\frac{H(D^+ u^n_j)-H(D^-u^n_j)}{D^+ u^n_j-D^-u^n_j}\right)  \right.\nonumber \right.\\
		&\quad\quad+\Delta x \left(\frac{h^M(D \,u^n_j,D^+ u^n_j)-h^M(D \,u^n_j,D^-u^n_j)}{D^+ u^n_j-D^-u^n_j}\right)\nonumber\\
		&\left.\quad\quad\left.-\Delta x\left(\frac{h^M(D^+ u^n_j,D \,u^n_j)-h^M(D^-u^n_j,D \,u^n_j)}{D^+ u^n_j-D^-u^n_j}\right)\right]\left(\frac{D^+ u^n_j-D^-u^n_j}{\Delta x}\right)\right|,\nonumber \\
	\end{align*}
whence we can conclude
	\begin{align}\label{stima_eps}
	\varepsilon^n
	&\leq K \left|(\Delta t L_H L_{H2} +2\Delta x L_{h^M})B \right| \nonumber \\
		&=KB\left(\lambda L_H L_{H2} +2L_{h^M}\right)\Delta x,
	\end{align}
where $L_{h^M}$ is the Lipschitz constant of $h^M$, whereas $L_H$ and $L_{H2}$ are the local Lipschitz constant of $H$ on $[-L_n,L_n]$ and $[-2L_n-\Delta t L_H B,2L_n+\Delta t L_HB]$, respectively. 
Notice that, if the smoothness indicators are computed using the definitions \eqref{beta_k_m}-\eqref{beta_k_p}, we have 
	$$
	\frac{\sqrt{\beta^+_0(u^n)_j}}{\Delta x}=\frac{D^+ u^n_j-D^-u^n_j}{\Delta x}=D^2u_j^n.
	$$
	Then, in such case by Lemma \ref{lemma_bound_B}, $x_j \in \mathcal R^n \Rightarrow D^2 u_j^n<B$, for some constant $B>0$ independent on $n$. Otherwise, we can obtain the same estimate by the definition \eqref{def_tilde_phi} of $\mathcal R^n$. 

Notice also that if the function $H$ is globally Lipschitz continuous we have the same estimate with $L_{H2}=L_H$, where now $L_H$ is the global Lipschitz constant of $H$. 
Consequently, the last statement would follow with $C=KB(\lambda L^2_H +2L_{h^M})$.
	
Let us now prove the main statement proceeding, as usual, by induction on $n\geq 0$ and noticing that it is sufficient to prove (\ref{lip_un}) for $i$ and $j$ such that $i=j\pm1$. 
	
For $n=0$, as we take $u^0_j=v_0(x_j)$ for $j\in \mathbb Z$, we have that (\ref{lip_un}) is satisfied by the Lipschitz continuity assumption on $v_0$, with constant $L_0$. 
	
Now, assuming that (\ref{lip_un}) is satisfied for $n-1>0$ so that $\varepsilon^k$ for $k=0,\dots,n-1$ are bounded by (\ref{stima_eps}), we can compute
	\begin{align*}
		\frac{|u_i^{n}-u_j^{n}|}{\Delta x}&=\frac{1}{\Delta x}\left|S^M(u^{n-1})_i+\phi_i \varepsilon^{n-1} \Delta t F(\cdot)_i - S^M(u^{n-1})_j-\phi_j \varepsilon^{n-1} \Delta t F(\cdot)_j\right| \\
		&\leq \frac{1}{\Delta x}\left(|S^M(u^{n-1})_i-S^M(u^{n-1})_j|+\varepsilon^{n-1}\Delta t|\phi_i F(\cdot)_i-\phi_j F(\cdot)_j|\right) \\
		&\leq \frac{|u^{n-1}_i-u^{n-1}_j|}{\Delta x} + \frac{2\Delta t}{\Delta x}\varepsilon^{n-1}
	\end{align*}
	then, iterating back and using the same arguments, 
	\begin{align*}
		\frac{|u_i^{n}-u_j^{n}|}{\Delta x}&\leq \frac{|u^{n-1}_i-u^{n-1}_j|}{\Delta x} + 2\Delta t C\leq \dots \\
		&\leq\frac{|u_i^{1}-u_j^{1}|}{\Delta x}  +2(n-1)\Delta t C 
		\leq \frac{|u_i^{0}-u_j^{0}|}{\Delta x} +2n\Delta t  C \\
		&\leq L_0 + 2\frac{T}{\Delta t}\Delta t C = L,
	\end{align*}
	where $C$ is well defined by (\ref{stima_eps}). Notice that we have used the nonexpansivity in $L^\infty$ of $S^M$ and the fact that $|F|\leq 1$, $|\phi|\leq 1$. 
%	\qed
\end{proof}
Therefore, it is clear that by construction our scheme is \emph{$\varepsilon$-monotone}, in the sense of the following
\begin{definition}[$\varepsilon$-monotonicity]
	A numerical scheme $S$ is \emph{$\varepsilon$-monotone} if for any functions $u,v$,
	$$
	\mbox{$u\leq v$ $\Rightarrow$  $S(u)\leq S(v)+C\varepsilon \Delta t$},
	$$
	where $C$ is constant and $\varepsilon \rightarrow 0 $ as $\Delta =(\Delta t, \Delta x) \rightarrow 0$.
\end{definition}
Thanks to that property, by applying the Barles-Souganidis result \cite{BS91}, the convergence follows directly. 
We conclude this section with the following theorem, which gives us the order of convergence for the Adaptive Filtered Schemes. 
\begin{theorem}\label{main_result}
Let the assumptions on $S^M$ and $S^A$ be satisfied.  Assume that $v_0$ and $H$ are Lipschitz continuous functions, $u_j^{n+1}$ is computed by (\ref{AFS})-(\ref{eps}), with $K>1/2$ and $\lambda=\frac{\Delta t}{\Delta x}$, a constant such that (\ref{cond_CFL}) is satisfied. 
Assume also that $\mathcal R^n$ is defined by (\ref{def_region_Rn}) or (\ref{def_tilde_phi}), with $\phi$ given by (\ref{phi_disc}).  
Let us denote by  $v_j^n := v(t^n,x_j)$ the values of the viscosity solution on the nodes of the grid. Then,
	\begin{enumerate}[i)]
		\item the AF scheme \eqref{AFS} satisfies Crandall-Lions estimate \cite{CL84}
		\begin{equation*}
			||u^n-v^n||_\infty \leq C_1\sqrt{\Delta x},\quad \forall \ n=0,\dots,N,
		\end{equation*}
		for some constant $C_1>0$ independent of  $\Delta x$. 
		\item (First order convergence for regular solutions) Moreover, if $v\in C^2([0,T]\times \mathbb R)$, then
		\begin{equation*}
			||u^n-v^n||_\infty \leq C_2\Delta x,\quad \forall \ n=0,\dots,N,
		\end{equation*}
		for some constant $C_2>0$ independent of $\Delta x$. 
		\item (High-order local consistency) Let $k\geq 2$ be the order of the scheme $S^A$. If  $v \in C^{l+1}$ in some neighborhood of a point $(t,x) \in [0,T]\times \mathbb R$, then for $1\leq l\leq k$,
		\begin{equation*}
			\mathcal E_{AF}(v^n)_j=\mathcal E_{A}(v^n)_j=O(\Delta x^l)+O(\Delta t^l)
		\end{equation*}
		for $t^n-t$, $x_j-x$, $\Delta t$, $\Delta x$ sufficiently small.
	\end{enumerate}
\end{theorem}
\begin{proof}
	\emph{\textbf{i)}} Let us proceed as has been done in \cite{BFS16} defining $w^{n+1}_j=S^M(w^n)_j$, the solution computed with the monotone scheme alone with $w^0_j=v_0(x_j)$. Then by definition,
	\begin{equation}
		u_j^{n+1}-w_j^{n+1}=S^M(u^n)_j-S^M(w^n)_j+\phi^n_j\varepsilon^n \Delta t F\left(\frac{S^A(u^n)_j-S^M(u^n)_j}{\varepsilon^n\Delta t}\right),
	\end{equation}
	whence, exploiting the nonexpansivity in $L^\infty$ of $S^M$, the definition of $\varepsilon^n$ and that $|F|\leq 1$,
	\begin{equation}
		\max_j|u_j^{n+1}-w_j^{n+1}|\leq \max_j|u^n_j-w^n_j|+\varepsilon^n \Delta t.
	\end{equation}
Then, proceeding recursively on $n\leq N$ and recalling that by Prop. \ref{prop_eps_n_un} there exists a constant $C>0$ such that $\varepsilon^n\leq C\Delta x:=\varepsilon$ for each $n$,
	\begin{equation}
		\label{1}
		\max_j|u^n_j-w^n_j|\leq \sum_{k=0}^{n-1} \varepsilon^k \Delta t\leq n\varepsilon \Delta t \leq T\varepsilon.
	\end{equation}
	At this point, by the triangular inequality
	\begin{equation}
		\max_j|u^{n+1}_j-v^{n+1}_j|\leq \max_j|u_j^{n+1}-w_j^{n+1}|+\max_j|w^{n+1}_j-v^{n+1}_j|,
	\end{equation}
	whence we have that
	\begin{equation}
		\max_j|u^{n+1}_j-v^{n+1}_j| \leq \max_j|w^n_j-v^n_j|+\varepsilon T \leq (C_{CL}+CT)\sqrt{\Delta x},
	\end{equation}
	with $C_{CL}>0$ given by the Crandall-Lions estimate for $S^M$.
	
	\noindent
	\emph{\textbf{ii)}} Let us recall that by (\ref{consistenza_M}), in the case of $v \in C^2$ the consistency error for the monotone scheme is such that $\mathcal E_{M}(v^n)_j\leq C_M(\Delta t+\Delta x)$. Then we can compute
	\begin{align*}
		|u_j^{n+1}-v_j^{n+1}|&=|S^M(u^n)_j+\phi_j\varepsilon^n \Delta t F(\cdot)-v^{n+1}_j|\\
		&\leq|S^M(u^n)_j-S^M(v^n)_j|+|S^M(v^n)_j-v^{n+1}_j|+\varepsilon^n\Delta t\\
		&\leq||u^n-v^n||_\infty +\Delta t\left(\mathcal E_{M}(v^n)+\varepsilon^n\right),
	\end{align*}
	whence, by recursion on $n\leq N$ and recalling what we have done in the previous point,
	\begin{equation}
		||u^n-v^n||_\infty\leq||u^0-v^0||_\infty+T\left(\max_{k=0,\dots,n-1}||\mathcal E_{M}(v^k)||_\infty+\varepsilon\right).
	\end{equation}
	To finish this proof what is left is to use the estimate on $\mathcal E_{M}$ and Prop. \ref{prop_eps_n_un} .  \\ 
	
	\noindent
	\emph{\textbf{iii)}} In order to show that $S^{AF}(v^n)_j=S^A(v^n)_j$ for $\Delta t$ e $\Delta x$ small enough it is sufficient to prove that
	\begin{equation}
		\frac{|S^A(v^n)_j-S^M(v^n)_j|}{\varepsilon^n \Delta t} \leq 1,\qquad \textrm{ for }(\Delta t,\Delta x)\to 0,
	\end{equation}
which follows directly from the computation we have done in Sect.  \ref{sect_tune_eps} for the tuning of the parameter $\ep^n$. In fact, if we plug (\ref{eps}) inside the previous inequality, we can deduce that
	\begin{equation*}
		\frac{|S^A(v^n)_j-S^M(v^n)_j|}{\varepsilon^n \Delta t}\leq \frac{1}{2K} +O(\Delta x)+O(\Delta t),
	\end{equation*}
	which, using that $K>1/2$ by assumption, leads to the thesis as $(\Delta t,\Delta x)\to 0$. Notice that we have used the property $\varepsilon^n=O(\Delta x)$ and exploited the CFL condition. 
%	\qed
\end{proof}

\begin{remark}\label{problema_lemma_bound_B}
	Notice that the assumption $M(\Delta x)=\frac{1}{2}-C\Delta x$, for some constant $C>0$ such that $M(\Delta x)>0$, needed to apply Lemma \ref{lemma_bound_B}, may give some problems in the proof of third assertion of the previous theorem. In fact, applying the standard definition (\ref{relazione_omega}) to the viscosity solution $v$ at a point $x_j$ and recalling the computations that led to (\ref{sviluppo_tau}), we get that
	$$
	\omega_j^\pm=\frac{1}{2}\mp \Delta x\frac{4 v_j^{\prime\prime}v_j^{\prime\prime\prime}}{(v_j^{\prime\prime})^2+\sigma}+O(\Delta x^2).
	$$
	Consequently, in order to be sure that if $v\in C^3$, then $j\in\mathcal R$, we have to choose the constant $C$ such that
	$$
	C\geq\left|\frac{4 v_j^{\prime\prime}v_j^{\prime\prime\prime}}{(v_j^{\prime\prime})^2+\sigma}\right|,
	$$
	or require additional smoothness assumptions on $v$, for example $v^{\prime \prime \prime}_j\ll v^{\prime\prime}_j$. This in fact poses a strong limitation on the applicability of Lemma \ref{lemma_bound_B}, at least in the present formulation. 
\end{remark}

%%%%%%%%%%%%%%%%%%%%%%%%%%%%%%%%%%%%%%%%%%%%%%
%%%%%%%%%%%%%%%  SECTION 5: Numerical tests
%%%%%%%%%%%%%%%%%%%%%%%%%%%%%%%%%%%%%%%%%%%%%%
\section{Numerical Tests}\label{sec:tests}
In this section we will present some one-dimensional examples designed to show the properties of our scheme, stated by Theorem \ref{main_result}. Our goal is also to compare the performances of our Adaptive Filtered Schemes $S^{AF}$ with those of the Filtered Scheme $S^F$ introduced in \cite{BFS16} and of the WENO scheme of second/third order of \cite{JP00}. 
Regarding the basic filtered scheme, we decided to avoid the introduction of the limiter used in \cite{BFS16} in all the numerical tests here presented for a more direct comparison. 
For all our numerical examples, we will use the function $\phi$ defined in \eqref{phi_disc}, with $\beta_k$ given by \eqref{beta_k_m}-\eqref{beta_k_p}, the mapping \eqref{map2}, and $M=0.1$, the parameter $\varepsilon^n$ defined in \eqref{eps}, and we will compute the errors and orders in  $L^\infty$ and $L^1$ norm. 
For each test,  we will specify the monotone and high-order schemes composing the filtered scheme.   
As already stated in Sect. \ref{subsec:filters}, in all our numerical simulations we will use the discontinuous filter function defined in \eqref{def_F1}. 
This choice is justified by comparison reasons, since in \cite{BFS16} this is the only filter function used and we suppose the authors in \cite{BFS16} used it since it gives the best performances for their $S^F$ scheme. Since our scheme is not sensitive to the choice of the filter function, we use the same as in \cite{BFS16} for best comparisons.  
At the end of the section, we will also show briefly how to use these schemes in order to approximate simple two dimensional problems. To be precise, in the following examples  we will refer to the standard CFL condition
\begin{equation}
\lambda \max|H_p(p)|\leq 1,
\end{equation}
to define $\lambda$ , which is alternative to (\ref{cond_CFL}) and more easily computed. \\
All the numerical tests have been implemented in language C++, with plots generated by using MATLAB. %and computation of the errors in MATLAB. 
The computer used for the simulations is a Notebook Asus F556U Intel Core i7-6500U with speed of 2.59 GHz and 12 GB of RAM. 

%%%%%%%%%%%%%%%%%%%%%%%%%%%%%%%%%%%%%%%%%%%%%%%
%%%%% ESEMPIO 1
{\bf Example 1: Transport equation.}
In order to test the capability of our scheme to handle both regular and singular regions, let us begin with a simple linear example and consider the problem
\begin{equation*}
\label{ex1:eq1}
\left\{
\begin{array}{l}
v_t(t,x)+v_x(t,x)=0\qquad \textrm{ in } (0,T)\times \Omega\\
v(0,x)=v_0(x),
\end{array}
\right.
\end{equation*}
with periodic boundary conditions, in two different situations. At first, aiming to test the full accuracy of the schemes, we consider the regular initial data (\emph{Case a}),
\begin{equation}
\label{ex1:eq2}
v_0(x)=\sin(\pi x), \qquad x\in \Omega
\end{equation}
with $\Omega=[-2,2]$ and $T=0.9$. Then, as a second test, we take the mixed initial datum (\emph{Case b}),
\begin{equation}
\label{ex1:eq3}
v_0(x)=\left\{
\begin{array}{ll}
\min\{(1-x)^2,(1+x)^2\}\qquad &\textrm{ if } -1\leq x \leq 1, \\
\sin^2(\pi (x-2))& \textrm{ if \ } 2\leq x\leq 3, \\
0 & \textrm{ otherwise},
\end{array}
\right.
\end{equation}
with $\Omega=[-1.5,3.5]$ and $T=2$. The latter problem models the transport of a function composed by two peaks, the first with one point of singularity  whereas the second is in $C^2$.
For these tests we use the \emph{Central Upwind scheme} (\ref{central_upwind}) as monotone scheme and the simple \emph{Heun-Centered} (HC) scheme (\ref{heun})-(\ref{RK2}) as high-order scheme, with $\lambda=0.9$ for \emph{Case a} and $\lambda=0.4$ for \emph{Case b}. We also compare the results obtained using $S^{AF}$ with the \emph{$4$th order Lax-Wendroff} scheme (\ref{lw_4ord}) as high-order scheme. We recall that the latter high-order scheme has a very compact $5$-points stencil, whereas the WENO scheme of second/third order (coupled with the \emph{third order Runge Kutta scheme}) has a stencil of nine points.
\begin{figure}[h!]
\includegraphics[width=0.45\textwidth]{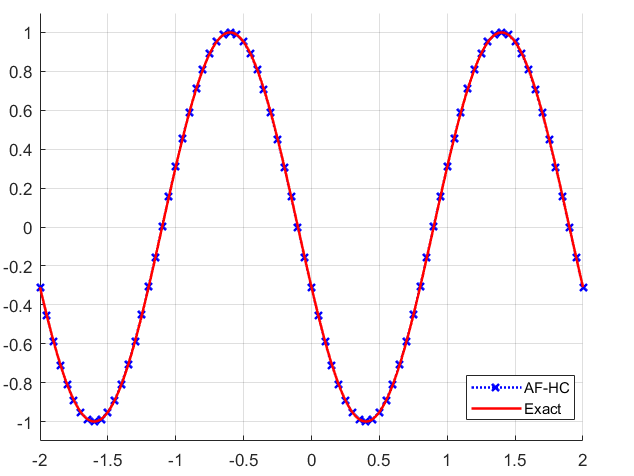}
\includegraphics[width=0.45\textwidth]{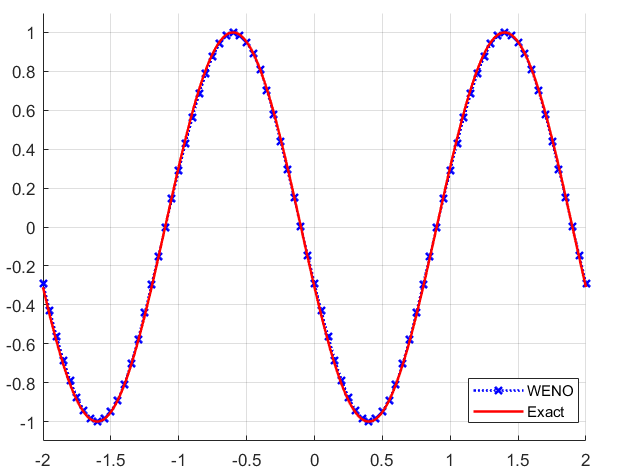}
\caption{\small{(Example 1a.) Plots at time $T=0.9$ with the AF-HC scheme on the left and WENO on the right for $\Delta x=0.05$. }\label{ex1:fig1}}
\end{figure}

% TABLE CON BETA JP
\begin{table}[h!]
\caption{(Example 1a.) Errors and orders in $L^\infty$ and  $L^1$ norms. \label{ex1:table1}}
\vspace{-0.3 cm}
\begin{center}
\hspace*{-0.3cm}
\begin{tabular}{c c|c c|c c|c c| c c}
& & \multicolumn{2}{|c|}{\bf F-HC ($5\Delta x$)}&\multicolumn{2}{|c|}{\bf AF-HC}& \multicolumn{2}{|c|}{\bf AF-LW4ord}&\multicolumn{2}{|c}{\bf WENO 2/3}\\
\hline
\hline
$N_x$ & $N_t$ &$L^\infty$ Err &Ord  & $L^\infty$ Err & Ord& $L^\infty$ Err & Ord & $L^\infty$ Err & Ord \\
\hline
$40$	& $10$  &$1.36$e-$02$ 	& &	$1.70$e-$02$	& &  $7.88$e-$03$ &	&	$8.02$e-$02$	\\

$80$	& $20$ & $2.56$e-$03$	& $2.41$ & $2.56$e-$03$	& $2.73$ & $8.66$e-$06$ &	$9.83$ & $2.62$e-$02$ &	$1.62$   \\

$160$	& $40$ & $5.76$e-$04$ &	$2.15$  & $5.76$e-$04$ &	$2.15$  & $5.43$e-$07$ &	$4.00$	& $4.50$e-$03$ & 	$2.54$  \\

$320$	& $80$ & $1.40$e-$04$ &	$2.04$	& $1.40$e-$04$ &	$2.04$ & $3.40$e-$08$ &	$4.00$ & $1.95$e-$04$ &	$4.52$ \\
\hline
\hline
$N_x$ & $N_t$ &$L^1$ Err &Ord  & $L^1$ Err & Ord& $L^1$ Err & Ord & $L^1$ Err & Ord \\
\hline
$40$	& $10$ &  $3.58$e-$02$	& & $3.29$e-$02$	& & $1.08$e-$02$	&	&	$2.07$e-$01$		\\

$80$	& $20$ & $6.66$e-$03$ &	$2.43$ & $6.66$e-$03$ &	$2.30$ & $2.25$e-$05$ &	$8.90$ & $4.14$e-$02$ & 	$2.32$   \\

$160$	& $40$ & $1.48$e-$03$ &	$2.17$  & $1.48$e-$03$ &	$2.17$  & $1.40$e-$06$ &	$4.01$	& $5.09$e-$03$ &	 $3.02$  \\

$320$	& $80$ & $3.57$e-$04$ &	$2.05$	 & $3.57$e-$04$ &	$2.05$  & $8.69$e-$08$ &	$4.01$ & $3.08$e-$04$ &	$4.05$ \\
\hline
\end{tabular}
\end{center}
\end{table}

In the first case (\emph{Case a}) of this test, all the schemes are very accurate and achieve optimal order in both norms, as shown in Tab. \ref{ex1:table1}. In this case, both filtered schemes have the same numerical results, except for a slight difference with the coarstest grid, and coincide with the simple HC high-order scheme, as expected (we avoided to add another column in the table to report also the results for the HC high-order scheme since they are the same). Moreover, we can see that our fourth order scheme is much more accurate even than the WENO scheme, despite the smaller stencil required. 
In Fig. \ref{ex1:fig1} we reported only the AF-HC scheme and the WENO scheme,
avoiding to show all the schemes since no differences are visible for that case. 

\begin{figure}[h!]
\includegraphics[width=0.5\textwidth]{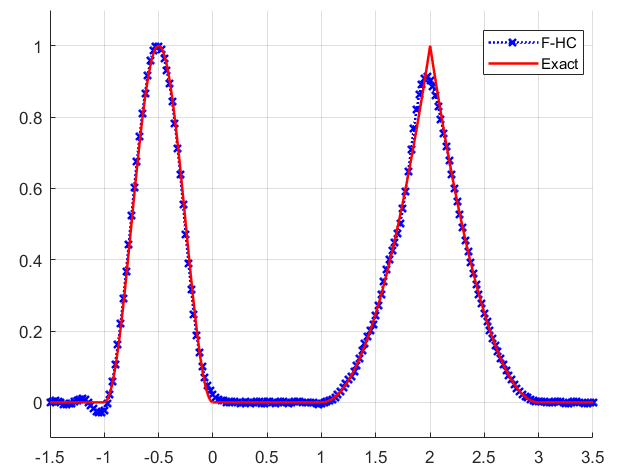}
\includegraphics[width=0.5\textwidth]{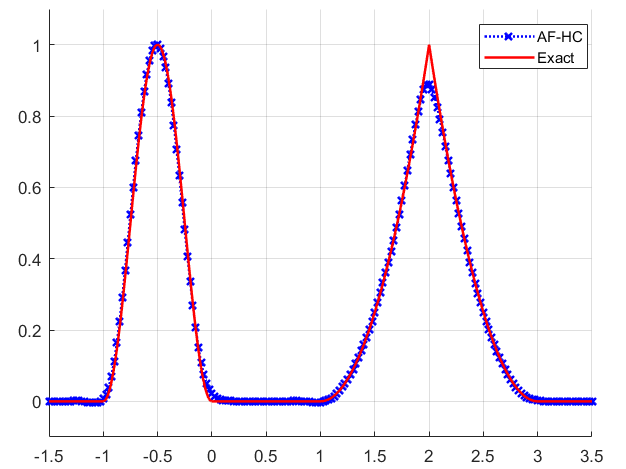}
\includegraphics[width=0.5\textwidth]{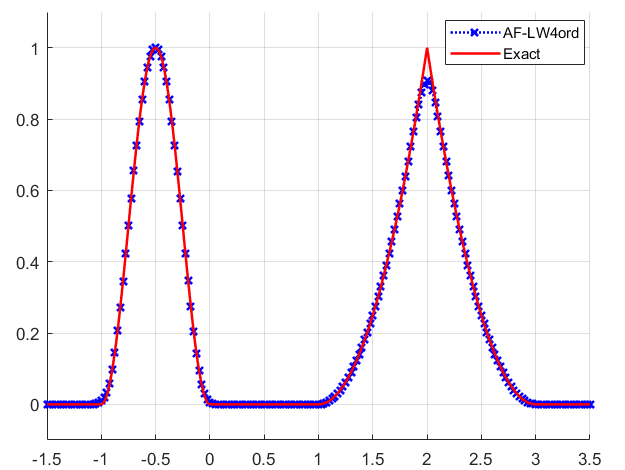}
\includegraphics[width=0.5\textwidth]{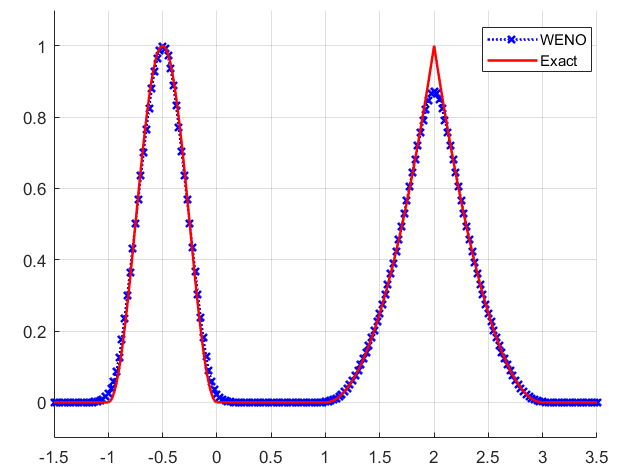}
\caption{\small{(Example 1b.) Plots of the solution at time $T=2$ with $\Delta x=0.025$. Top: simple filtered scheme with HC on the left, adaptive on the right. Bottom: fourth order AF scheme on the left and WENO on the right. }\label{ex1:fig2}}
\end{figure}

% TABLE CON BETA JP
\begin{table}[h!]
\caption{(Example 1b.) Errors and orders in $L^\infty$ and  $L^1$ norms. \label{ex1:table2}}
\vspace{-0.3 cm}
\begin{center}
\hspace*{-0.3cm}
\begin{tabular}{c c|c c|c c|c c|c c}
& & \multicolumn{2}{|c|}{\bf F-HC ($10\Delta x$)}&\multicolumn{2}{|c|}{\bf AF-HC}& \multicolumn{2}{|c|}{\bf AF-LW4ord}&\multicolumn{2}{|c}{\bf WENO 2/3}\\
\hline
\hline
$N_x$ & $N_t$ &$L^\infty$ Err &Ord  & $L^\infty$ Err & Ord& $L^\infty$ Err & Ord & $L^\infty$ Err & Ord  \\
\hline
$50$	& $50$ &       $3.46$e-$01$		&	&	$3.55$e-$01$ &	& $3.31$e-$01$	& &	$3.47$e-$01$		\\

$100$	& $100$  & $1.41$e-$01$	& $1.29$ & $1.90$e-$01$	& $0.90$ & $1.72$e-$01$ &	$0.94$  & $2.07$e-$01$	& $0.75$ \\

$200$	& $200$ &  $9.69$e-$02$	& $0.54$	& $1.17$e-$01$	& $0.70$ & $9.72$e-$02$ &	$0.82$ & $1.28$e-$01$	& $0.70$ \\

$400$	& $400$ &  $7.29$e-$02$	& $0.41$ & $7.27$e-$02$	& $0.69$   & $5.47$e-$02$ &	$0.83$	& $7.66$e-$02$	& $0.74$ \\
\hline
\hline
$N_x$ & $N_t$ &$L^1$ Err &Ord  & $L^1$ Err & Ord& $L^1$ Err & Ord & $L^1$ Err & Ord \\
\hline
$50$	& $50$ &   $4.34$e-$01$		&	&	$3.31$e-$01$ 	&	& $2.68$e-$01$	& &	$3.62$e-$01$		\\

$100$	& $100$ &  $1.41$e-$01$	& $1.63$ & $1.19$e-$01$	& 	$1.47$  & $9.27$e-$02$	& 	$1.53$ & $1.39$e-$01$	& $1.39$ \\

$200$	& $200$ &  $4.24$e-$02$	& $1.73$	& $3.03$e-$02$	& 	$1.98$  & $1.30$e-$02$	& 	$2.83$ & $3.83$e-$02$	& $1.86$ \\

$400$	& $400$ &  $1.38$e-$02$	& $1.62$ & $9.51$e-$03$	& 	$1.67$  & $3.07$e-$03$	& 	$2.08$ & $8.39$e-$03$	& $2.19$ \\
\hline
\end{tabular}
\end{center}
\end{table}

\begin{table}[h!]
\caption{(Example 1b.) CPU times in seconds. \label{ex1:table3}}
\centering
\begin{tabular}{c c| p{1.8cm}  | p{1.8cm} | p{1.8cm}  | c }
$N_x$ & $N_t$ & \centering {\bf  F-HC } & \centering {\bf AF-HC }& \centering {\bf AF-LW4ord}& {\bf WENO 2/3}\\
\hline
\hline
$50$	& $50$ &\centering 	$0.000\ s$ 	&\centering 	$0.001\ s$ 	&\centering $0.001\ s$  &$0.002\ s$   \\

$100$	& $100$ &\centering 	$0.001\ s$  &\centering 	$0.004\ s$   &\centering $0.005\ s$  	& $0.006\ s$    \\

$200$ & $200$ &\centering 	$0.004\ s$  &\centering  	$0.016\ s$   &\centering $0.018\ s$   	& $0.026\ s$ \\

$400$ &	$400$ &\centering 	$0.019\ s$  &\centering  	$0.061\ s$   &\centering $0.077\ s$   	 & $0.095\ s$ 	\\
\hline
\end{tabular}
\end{table}

For the second case (\emph{Case b}), looking at Fig. \ref{ex1:fig2} we can observe that  the adaptive tuning of $\varepsilon^n$ is able to contain the oscillations behind the peaks produced by the unstable HC scheme, which are clearly visible instead in the case of $S^F$ with $\varepsilon=10\Delta x$. We can also see that our  scheme coupled with the fourth order scheme produces again almost always the best results in terms of errors and orders in both norms (see Tab. \ref{ex1:table2}) and gives the best resolution of the peaks, preserving better the kink of the singularity and the feet of the regular part, without introducing any oscillation. In Tab.  \ref{ex1:table3} we reported the CPU times for this \emph{Case b}, in which the evolution lasts longer. All the schemes are very fast and complete the computations in less than $0.1$ s for all the refinements. Note that our two adaptive filtered schemes perform faster than the WENO scheme, even in the case of the fourth-order scheme. On the other hand, as could be expected, the adaptive procedure increases the cost of the filtering process three/four times depending on the refinement with respect to the basic filtered scheme. 

%%%%%%%%%%%%%%%%%%%%%%%%%%%%%%%%%%%%%%%%%%%%%%%
{\bf Example 2: Eikonal equation.} As a first nonlinear problem let us consider the eikonal equation
\begin{equation}
\label{ex2:eq1}
\left\{
\begin{array}{l}
v_t(t,x)+|v_x(t,x)|=0\qquad \textrm{ in } (0,0.3)\times(-2,2), \\
v_0(x)=\max\{1-x^2,0\}^4,
\end{array}
\right.
\end{equation}
where $v_0$ is a Lipschitz continuous initial datum with high regularity (\emph{Case a}). Then, we repeat the simulation with the ``reversed'' initial datum (\emph{Case b}) 
\begin{equation}
\label{ex2:eq2}
v_0(x)=-\max\{1-x^2,0\}^4,
\end{equation}
which presents also a major problem in the origin because of the saddle point in the hamiltonian, where two directions of propagation occur. Here the aim is mainly to compare the results obtained by the unfiltered high-order schemes with their filtered versions, in order to show the stabilization property of the filtering process. For the monotone scheme we use the numerical hamiltonian (\ref{hm_eik}), whereas to achieve high-order we use the \emph{Lax-Wendroff-Richtmyer (LWR)} scheme (\ref{lwr}). Moreover, as in the previous example, we present also the results obtained with the AF scheme coupled with the fourth order LW scheme. The CFL number is set to $0.375$ for both simulations.
\begin{figure}[h!]
\includegraphics[width=0.327\textwidth]{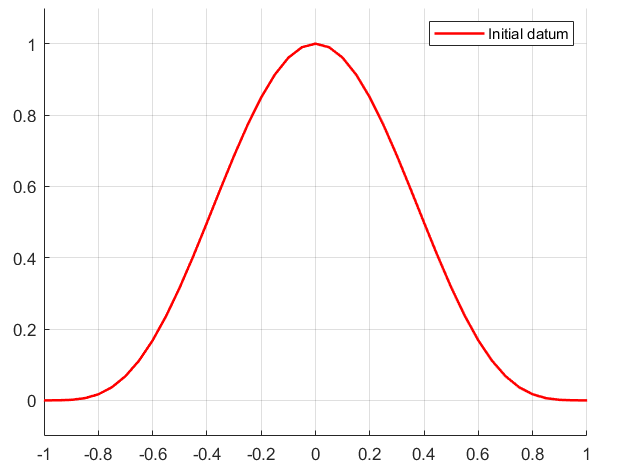}
\includegraphics[width=0.327\textwidth]{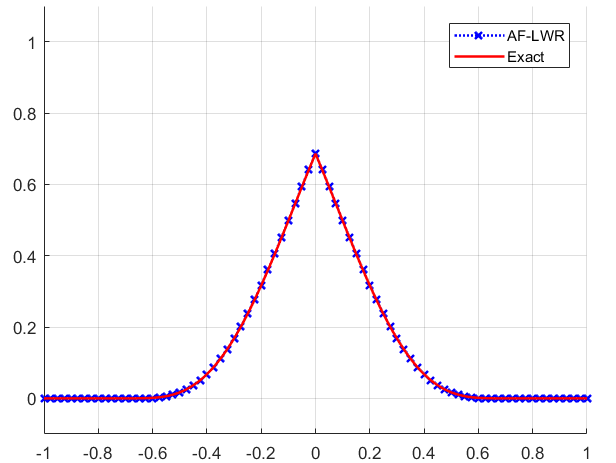}
\includegraphics[width=0.327\textwidth]{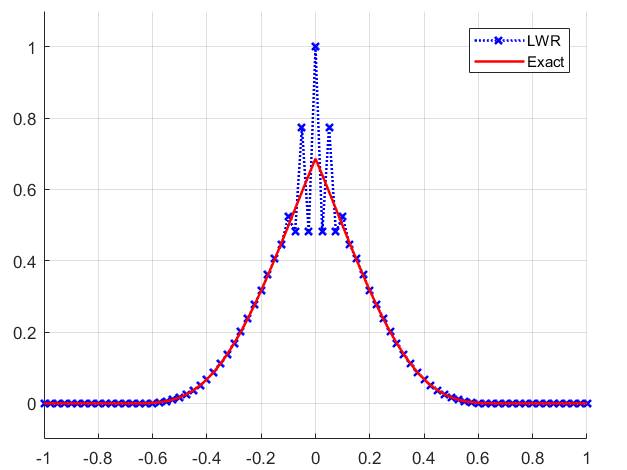}
\caption{\small{(Example 2a.) Initial datum (left) and plots of the solution at time $T=0.3$ with the AF scheme (center) and the LWR scheme (right) for $\Delta x=0.025$.}\label{ex2:fig1}}
\end{figure}

% TABLE CON BETA JP
\begin{table}[h!]
\caption{(Example 2a.) Errors and orders in $L^\infty$ and  $L^1$ norms. \label{ex2:table1}}
\vspace{-0.3 cm}
\begin{center}
\hspace*{-0.3cm}
\begin{tabular}{c c|c c|c c|c c| c c}
& &\multicolumn{2}{|c|}{\bf F-LWR ($5\Delta x$)}&\multicolumn{2}{|c|}{\bf AF-LWR}& \multicolumn{2}{|c|}{\bf AF-LW4ord}&\multicolumn{2}{|c}{\bf WENO 2/3}\\
\hline
\hline
$N_x$ & $N_t$ &$L^\infty$ Err &Ord  & $L^\infty$ Err & Ord& $L^\infty$ Err & Ord & $L^\infty$ Err & Ord \\
\hline
$40$	& $8$ &	$1.96$e-$02$	& & $1.89$e-$02$ & &  $1.95$e-$02$	& & $6.81$e-$02$			\\

$80$	& $16$  & $4.48$e-$03$	& $2.13$ & $3.56$e-$03$ &	$2.41$   & $1.04$e-$02$ &	$0.90$ & $3.42$e-$02$&	$1.00$  \\

$160$	& $32$  & $1.06$e-$03$	& $2.08$ & $8.53$e-$04$ &	$2.06$ & $1.45$e-$03$ &	$2.85$	& $1.62$e-$02$&	$1.08$  \\

$320$	& $64$  & $2.56$e-$04$	& $2.05$	& $2.20$e-$04$ &	$1.96$ & $2.31$e-$04$ &	$2.65$ & $7.52$e-$03$&	$1.11$  \\
\hline
\hline
$N_x$ & $N_t$ &$L^1$ Err &Ord  & $L^1$ Err & Ord& $L^1$ Err & Ord & $L^1$ Err & Ord \\
\hline
$40$	& $8$ &	$1.52$e-$02$	& &  $1.63$e-$02$	& &  $1.28$e-$02$  & &  $2.05$e-$02$\\

$80$	& $16$  & $3.78$e-$03$	& $2.01$  & $3.61$e-$03$	& 	$2.17$ & $1.11$e-$03$	& 	$3.53$   & $4.68$e-$03$	&$2.13$\\

$160$	& $32$   & $8.94$e-$04$	& $2.08$ & $8.80$e-$04$	& 	$2.04$ & $7.48$e-$05$	& 	$3.89$	& $9.55$e-$04$	&$2.29$ \\

$320$	& $64$  & $2.09$e-$04$	& $2.09$	& $2.08$e-$04$	& 	$2.08$ & $7.14$e-$06$	& 	$3.39$ & $1.40$e-$04$	&$2.78$ \\
\hline
\end{tabular}
\end{center}
\end{table}

Let us first point out that, as Figs. \ref{ex2:fig1} - \ref{ex2:fig2} clearly show, the LWR scheme is unstable in the origin in both situations, whereas the AF scheme (and the simple filtered scheme) is stable. 
Then, for the first case, looking at Tab. \ref{ex2:table1} we can see that the filtered-LWR schemes  give almost the same results, are of high-order in both norms and get lower errors with respect to the WENO scheme in almost all simulations. 
Moreover, we can recognize the typical improvements and drawbacks of the fourth order LW scheme, which has a slightly wider stencil. In fact, as will be shown also in the following examples, the scheme has bigger errors in the $L^\infty$ norm  with respect to the second order AF scheme whereas has way better errors and orders in the $L^1$ norm, achieving almost optimal order, which testifies the overall improvement. 

\begin{figure}[h!]
\includegraphics[width=0.45\textwidth]{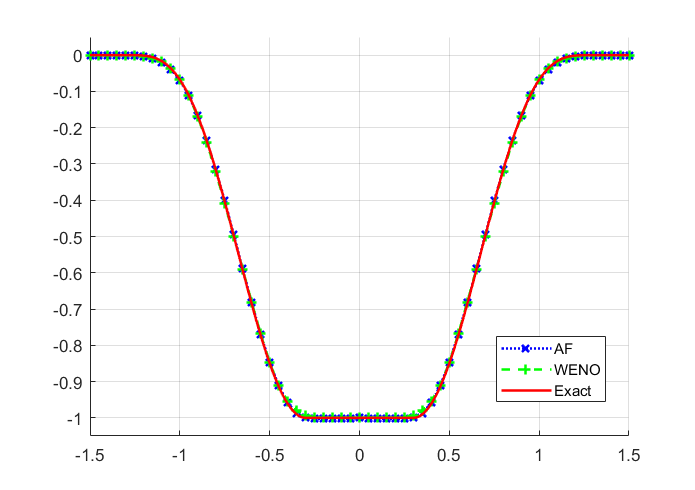}
\includegraphics[width=0.45\textwidth]{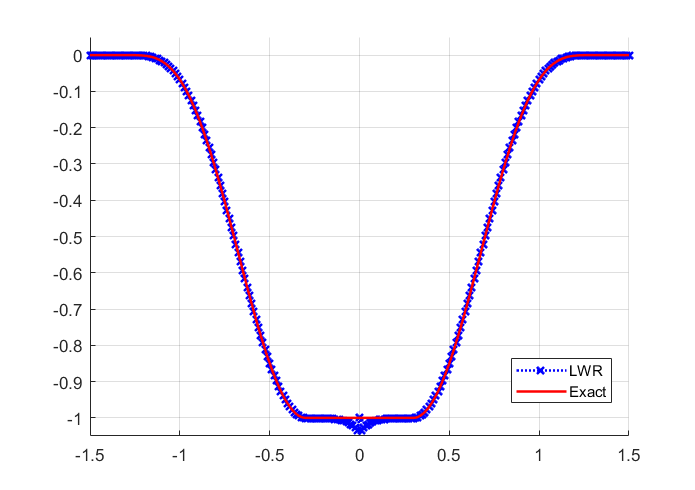}
\caption{\small{(Example 2b.) Plots at time $T=0.3$ with the AF and WENO schemes with $\Delta x=0.05$ (left) and LWR scheme with $\Delta x=0.0125$ (right).}\label{ex2:fig2}}
\end{figure}

% TABLE CON BETA JP
\begin{table}[h!]
\caption{(Example 2b.) Errors and orders in $L^\infty$ and  $L^1$ norms. \label{ex2:table2}}
\vspace{-0.3 cm}
\begin{center}
\hspace*{-0.3cm}
\begin{tabular}{c c|c c|c c|c c| c c}
& &\multicolumn{2}{|c|}{\bf F-LWR ($5\Delta x$)}&\multicolumn{2}{|c|}{\bf AF-LWR}& \multicolumn{2}{|c|}{\bf AF-LW4ord}&\multicolumn{2}{|c}{\bf WENO 2/3}\\
\hline
\hline
$N_x$ & $N_t$ &$L^\infty$ Err &Ord  & $L^\infty$ Err & Ord& $L^\infty$ Err & Ord & $L^\infty$ Err & Ord \\
\hline
$40$	& $8$ &	$1.91$e-$02$	& & $2.35$e-$02$ &  &  $2.42$e-$02$	& & $2.33$e-$02$			\\

$80$	& $16$  & $9.24$e-$03$	& $1.04$ & $3.37$e-$03$ &	$2.80$  & $7.51$e-$03$	& $1.69$  & $1.02$e-$02$&	$1.19$ \\

$160$	& $32$  & $5.77$e-$03$ & $0.68$ & $1.58$e-$03$ &	$1.09$ & $2.14$e-$03$ &	$1.81$	& $4.10$e-$03$&	$1.32$  \\

$320$	& $64$  & $3.46$e-$03$	& $0.74$	& $7.09$e-$04$ &	$1.16$ & $6.92$e-$04$	& $1.63$  & $1.22$e-$03$&	$1.75$\\
\hline
\hline
$N_x$ & $N_t$ &$L^1$ Err &Ord  & $L^1$ Err & Ord& $L^1$ Err & Ord & $L^1$ Err & Ord \\
\hline
$40$	& $8$ 	&	$2.38$e-$02$	& & $2.24$e-$02$	& &  $2.28$e-$02$  & &  $2.96$e-$02$\\

$80$	& $16$  & $8.48$e-$03$	& $1.49$  & $5.70$e-$03$	& $1.98$ & $2.05$e-$03$	& $3.48$   & $7.04$e-$03$	&$2.07$\\

$160$	& $32$  & $3.41$e-$03$	& $1.32$ &  $1.82$e-$03$ &	$1.65$  & $3.20$e-$04$ &	$2.68$ & $1.43$e-$03$	&$2.30$ \\

$320$	& $64$  & $1.52$e-$03$	& $1.17$	& $5.84$e-$04$	& $1.64$ & $6.38$e-$05$	& $2.33$  & $2.82$e-$04$	&$2.34$\\
\hline
\end{tabular}
\end{center}
\end{table}

For \emph{Case b}, looking at Tab. \ref{ex2:table2} we can repeat almost the same considerations made for \emph{Case a}, but this time the improvements given by the adaptive filtering are more evident. The AF-LWR scheme is again of high-order especially in the $L^1$ norm, without the need to introduce any limiter as has been done in \cite{BFS16}, and the numerical results are always comparable to those obtained by the WENO scheme of second/third order, whereas the AF-LW4ord scheme produces again bigger errors in $L^\infty$ with respect to the second-order AF-LWR scheme and better orders in the $L^1$ norm. 

%%%%%%%%%%%%%%%%%%%%%%%%%%%%%%%%%%%%%%%%%%%%%%%
{\bf Example 3: Burgers' equation.} Let us consider now the Burgers' equation for HJ with a regular initial datum 
\begin{equation}
\label{ex3:eq1}
\left\{
\begin{array}{l}
v_t(t,x)+\frac{1}{2}(v_x(t,x)+1)^2=0\qquad \textrm{ in } (0,T)\times(0,2), \\
v_0(x)=-\cos(\pi x),
\end{array}
\right.
\end{equation}
which is a test case widely used in literature. In order to test the full accuracy of the schemes even in the nonlinear case, we first run the simulation for $T=\frac{4}{5\pi^2}$, when the solution is still regular, with $\lambda=\frac{2}{\pi^2}\approx 0.2<{\max|H_p|}^{-1}=0.5$. Then, we consider the final time $T=\frac{3}{2\pi^2}$ when a moving (to the right) singularity appears, taking $\lambda=\frac{15}{8\pi^2}\approx 0.19$. For both simulations we use the the Central Upwind monotone scheme and the LWR scheme for both the filtered schemes and compare the results as before with the WENO scheme and the fourth order AF scheme. 
In Fig. \ref{ex3:fig1} we report the intial datum of the problem and the solution produced by the AF-LWR scheme at the two different times in order to show the different behavior. 
\begin{figure}[h!]
	\includegraphics[width=0.327\textwidth]{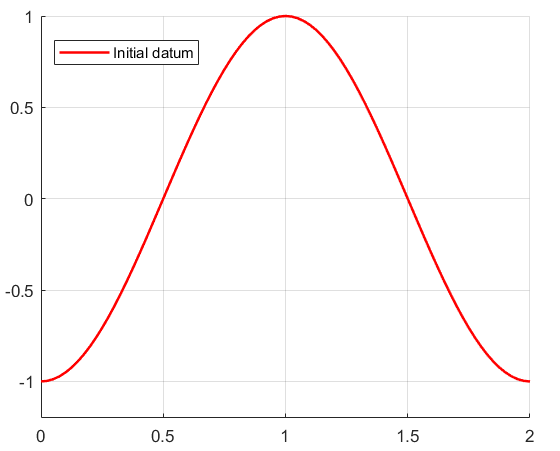}
	\includegraphics[width=0.327\textwidth]{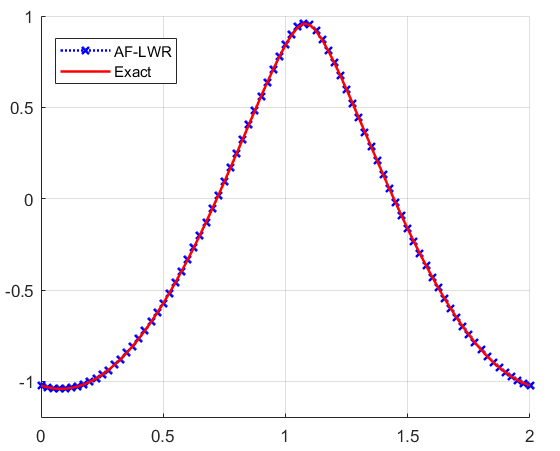}
	\includegraphics[width=0.327\textwidth]{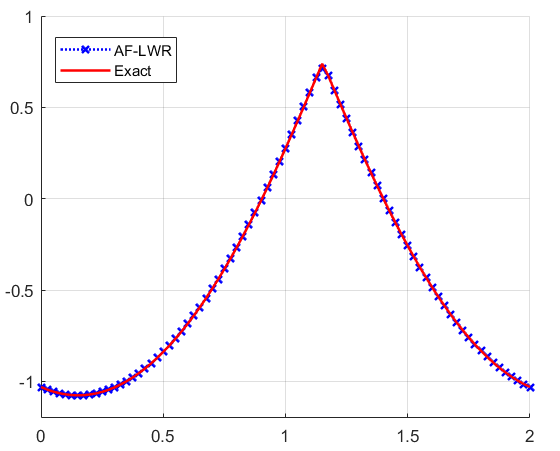}
	\caption{\small{(Example 3.) From left to right: initial datum of problem (\ref{ex3:eq1}) and plots of the solution with AF-LWR at time $T=4/(5\pi^2)$ and $T=3/(2\pi^2)$ for $\Delta x=0.025$.}\label{ex3:fig1}}
\end{figure}

This example summarizes all the behaviors already seen in the previous cases. In fact, as displayed by Tabs. \ref{ex3:table1}-\ref{ex3:table2}, if the solution is still regular the fourth order AF scheme gives the best results and achieves the optimal order in both norms, whereas when the singularity appears, it gets bigger errors in $L^\infty$ norm but  lower errors and better orders in the $L^1$ norm with respect to the second order filtered schemes. Here we have to notice that the WENO scheme has better errors and orders in the second simulation with respect to all the filtered schemes. 
Moreover, we can clearly see that the simple filtered scheme depends heavily on the choice of $\ep$, in fact after extensive computations we noticed that choosing for example $\ep=5\Delta x$ we get worse results in both cases, whereas if we increase the constant we get better results in the regular case and worse in the latter. In the tables we presented the results for the choice that gives the best results in the singular case, whereas  
it has clearly problems in the first situation. This is the main advantage of the adaptive $\ep^n$ which is able to tune itself in the right way depending on the local (in time) regularity of the solution. 

% TABLE CON BETA JP
\begin{table}[h!]
\caption{(Example 3.) $T=4/(5\pi^2)$. Errors and orders in $L^\infty$ and  $L^1$ norms. \label{ex3:table1}}
\vspace{-0.3 cm}
\begin{center}
\hspace*{-0.3cm}
\begin{tabular}{c c|c c|c c|c c| c c}
& &\multicolumn{2}{|c|}{\bf F-LWR ($10\Delta x$)}&\multicolumn{2}{|c|}{\bf AF-LWR}& \multicolumn{2}{|c|}{\bf AF-LW4ord}&\multicolumn{2}{|c}{\bf WENO 2/3}\\
\hline
\hline
$N_x$ & $N_t$ &$L^\infty$ Err &Ord  & $L^\infty$ Err & Ord& $L^\infty$ Err & Ord & $L^\infty$ Err & Ord \\
\hline
$40$	& $8$ &	$1.30$e-$02$	& & $9.61$e-$03$ & &  $1.89$e-$03$	& & $1.04$e-$02$			\\

$80$	& $16$  & $8.67$e-$03$	& $0.59$ & $2.77$e-$03$	& $1.79$ & $2.84$e-$04$	& $2.73$   & $2.12$e-$03$	& $2.30$ \\

$160$	& $32$  &$5.07$e-$03$	& $0.77$ & $7.24$e-$04$	& $1.94$ & $2.68$e-$05$	& $3.41$	& $1.82$e-$04$	& $3.54$  \\

$320$	& $64$  & $2.66$e-$03$	& $0.93$	& $1.83$e-$04$	& $1.99$ & $1.89$e-$06$	& $3.83$  & $2.05$e-$05$	& $3.15$\\
\hline
\hline
$N_x$ & $N_t$ &$L^1$ Err &Ord  & $L^1$ Err & Ord& $L^1$ Err & Ord & $L^1$ Err & Ord \\
\hline
$40$	& $8$ &	$3.76$e-$03$	& & $3.13$e-$03$	& &  $3.31$e-$04$  & &  $3.67$e-$03$\\

$80$	& $16$  & $1.29$e-$03$	& $1.54$  &$8.20$e-$04$	& $1.93$ & $1.85$e-$05$	& $4.16$  & $6.57$e-$04$	& $2.48$\\

$160$	& $32$   & $4.49$e-$04$	& $1.52$ & $2.04$e-$04$	& $2.01$ & $1.43$e-$06$	& $3.70$	& $5.43$e-$05$	& $3.60$ \\

$320$	& $64$  & $1.82$e-$04$	& $1.30$	&  $5.09$e-$05$	& $2.00$ & $9.80$e-$08$	&$3.86$ & $2.98$e-$06$	& $4.19$\\
\hline
\end{tabular}
\end{center}
\end{table}

% TABLE CON BETA JP
\begin{table}[h!]
\caption{(Example 3.) $T=3/(2\pi^2)$. Errors and orders in $L^\infty$ and  $L^1$ norms. \label{ex3:table2}}
\vspace{-0.3 cm}
\begin{center}
\hspace*{-0.3cm}
\begin{tabular}{c c|c c|c c|c c| c c}
& &\multicolumn{2}{|c|}{\bf F-LWR ($10\Delta x$)}&\multicolumn{2}{|c|}{\bf AF-LWR}& \multicolumn{2}{|c|}{\bf AF-LW4ord}&\multicolumn{2}{|c}{\bf WENO 2/3}\\
\hline
\hline
$N_x$ & $N_t$ &$L^\infty$ Err &Ord  & $L^\infty$ Err & Ord& $L^\infty$ Err & Ord & $L^\infty$ Err & Ord \\
\hline
$40$	& $16$ 	&	$4.88$e-$02$	& &$5.53$e-$02$ & &  $5.86$e-$02$ & & $3.89$e-$02$			\\

$80$	& $32$  & $2.47$e-$02$	& $0.98$ &$2.50$e-$02$	& $1.15$ &$2.62$e-$02$	& $1.16$   & $1.61$e-$02$	& $1.27$ \\

$160$ & $64$  &$9.81$e-$03$	& $1.33$ &  $9.99$e-$03$	& $1.32$ &$1.03$e-$02$	& $1.34$	& $5.12$e-$03$	& $1.65$  \\

$320$ & $128$  & $2.57$e-$03$	& $1.93$ & $2.59$e-$03$	& $1.95$ & $2.67$e-$03$	&  $1.95$   & $8.40$e-$04$ & $2.61$\\
\hline
\hline
$N_x$ & $N_t$ &$L^1$ Err &Ord  & $L^1$ Err & Ord& $L^1$ Err & Ord & $L^1$ Err & Ord \\
\hline
$40$	& $16$ 	&	$5.17$e-$03$	& & $5.38$e-$03$	& & $3.18$e-$03$  & &  $3.69$e-$03$\\

$80$	& $32$  &$1.26$e-$03$	& $2.03$  & $1.28$e-$03$	& $2.08$ & $6.73$e-$04$	& $2.24$   & $6.94$e-$04$	& $2.41$\\

$160$ & $64$   & $2.86$e-$04$	& $2.14$ & $2.87$e-$04$	& $2.15$ & $1.31$e-$04$	&  $2.36$	& $8.67$e-$05$	& $3.00$ \\

$320$ & $128$  & $5.68$e-$05$ & 	$2.33$	&  $5.68$e-$05$ & 	$2.34$	& $1.70$e-$05$	&$2.95$ & $6.40$e-$06$	& $3.76$\\
\hline
\end{tabular}
\end{center}
\end{table}

%%%%%%%%%%%%%%%%%%%%%%%%%
\begin{figure}[h!]
	\includegraphics[width=0.48\textwidth]{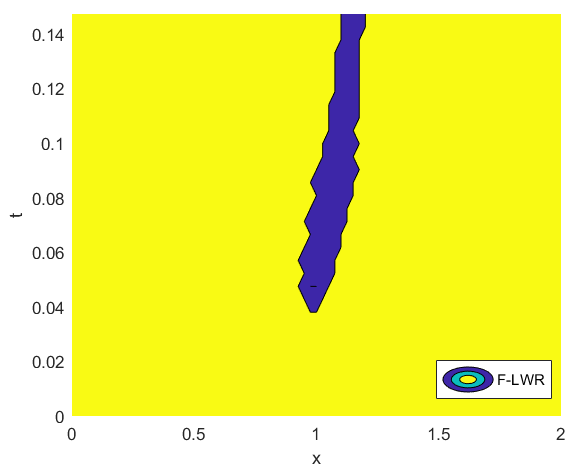}
	\includegraphics[width=0.48\textwidth]{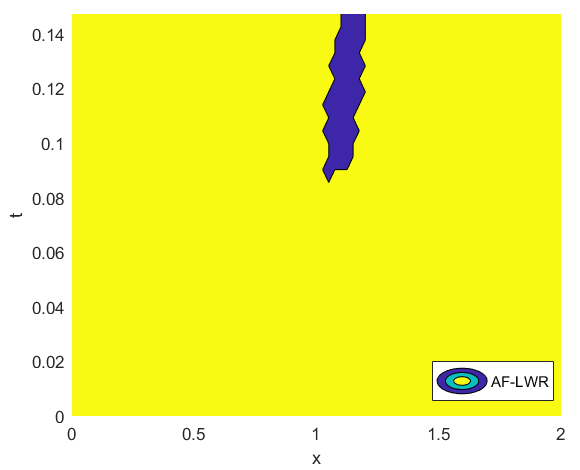}
	\caption{\small{(Example 3.) Regions of activity of $S^M$ (blue) and $S^A$ (yellow) for the F-LWR scheme (left) and the AF-LWR scheme (right) with $\Delta x=0.025$.}\label{ex3:fig2}}
\end{figure}

\begin{table}[h!]
	\caption{(Example 3.) $T=3/(2\pi^2)$. CPU times in seconds. \label{ex3:table3}}
	\centering
	\begin{tabular}{c c| p{1.8cm}  | p{1.8cm} | p{1.8cm}  | c }
		$N_x$ & $N_t$ & \centering {\bf  F-LWR } & \centering {\bf AF-LWR }& \centering {\bf AF-LW4ord}& {\bf WENO 2/3}\\
		\hline
		\hline
		$40$	& $16$ &\centering 	$0.000\ s$ 	&\centering 	$0.000\ s$ 	&\centering $0.001\ s$  &$0.000\ s$   \\
		
		$80$	& $32$ &\centering 	$0.000\ s$  &\centering 	$0.001\ s$   &\centering $0.002\ s$  	& $0.001\ s$    \\
		
		$160$ & $64$ &\centering 	$0.001\ s$  &\centering  	$0.004\ s$   &\centering $0.005\ s$   	& $0.004\ s$ \\
		
		$320$ &	$128$ &\centering 	$0.005\ s$  &\centering  	$0.016\ s$   &\centering $0.020\ s$   	 & $0.016\ s$ 	\\
		\hline
	\end{tabular}
\end{table}

In order to give a visual evidence of that latter property, in Fig.  \ref{ex3:fig2} we reported the regions of activity of the schemes composing the two second order filtered schemes. There we can clearly see that our procedure is able to better localize the presence of the singularity, whereas when the solution is still regular the high-order scheme is always active. On the other hand, if we look at the computational times in Tab.  \ref{ex3:table3}, we can see that the basic filtered scheme is of course the fastest scheme, whereas the other three schemes have very similar CPU times, with the fourth order scheme performing slightly slower. 

%%%%%%%%%%%%%%%%%%%%%%%%%%%%%%%%%%%%%%%%%
{\bf Example 4: Nonconvex Hamiltonian.}
In this example we consider a well known test case for nonconvex Hamiltonians (see e.g. \cite{JP00}), that is
\begin{equation*}
\label{ex1n:eq1n}
\left\{
\begin{array}{ll}
v_t(t,x)-\cos (v_x(t,x)+1)=0\qquad \textrm{ in } (0,T)\times (-1,1)\\
v(0,x)=-\cos(\pi x),
\end{array}
\right.
\end{equation*}
with periodic boundary conditions and final time $T=3/(2\pi^2)$, when two singularities appear in the solution, as can be seen in Fig. \ref{ex1n:fig1}. \begin{figure}[h!]
\includegraphics[width=0.45\textwidth]{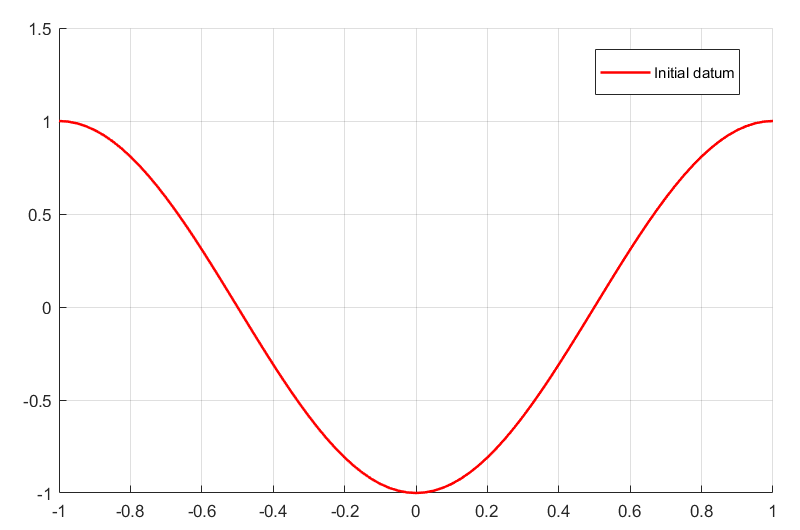} % New_Test/
\includegraphics[width=0.45\textwidth]{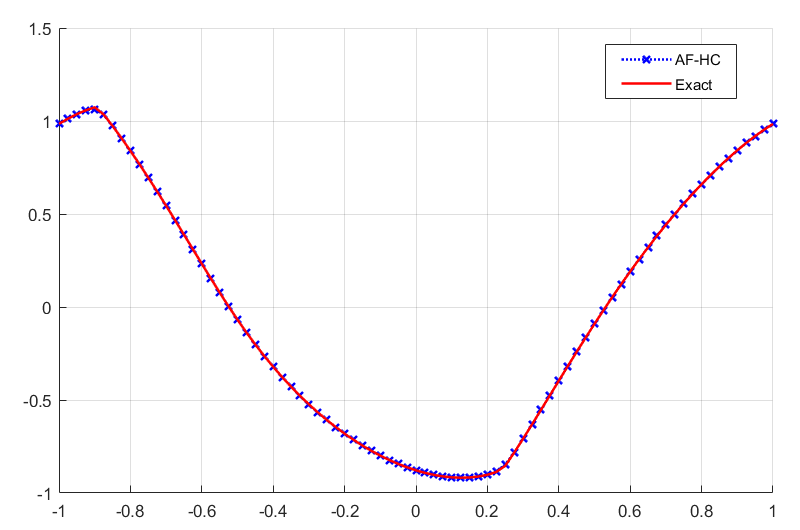} % New_Test/
\caption{\small{(Example 4.) Initial datum on the left and plots of the exact and the AF-HC solutions at time $T=3/(2\pi^2)$ for $\Delta x=0.025$ on the right. }\label{ex1n:fig1}}
\end{figure}
In order to define the monotone scheme for this test, we use the Lax-Friedrichs hamiltonian \eqref{eq:LF_hamiltonian} with $\theta=1$ as in \cite{BFS16}, whereas the CFL number is set to $0.31$. A reference solution is computed by using the AF-LW4ord scheme with $10240$ points. \\
We reported the global errors in Tab. \ref{ex1n:table1} and the errors far away from singular points (regularity region) in Tab. \ref{ex1n:table2}. More precisely, for the second table we consider the set of grid points $x$ such that $|x-x_i|\geq 0.05$, for $i=1,2$, where $x_1=-0.895$ and $x_2=0.245$ are approximately the position of the singularities. 
%
% TABLE CON BETA JP
\begin{table}[h!]
\caption{(Example 4.) Global Errors and orders in $L^\infty$ and  $L^1$ norms. \label{ex1n:table1}}
\vspace{-0.3 cm}
\begin{center}
\hspace*{-0.3cm}
\begin{tabular}{c c|c c|c c|c c| c c}
& & \multicolumn{2}{|c|}{\bf F-HC ($5\Delta x$)}&\multicolumn{2}{|c|}{\bf AF-HC}& \multicolumn{2}{|c|}{\bf AF-LW4ord}&\multicolumn{2}{|c}{\bf WENO 2/3}\\
\hline
\hline
$N_x$ & $N_t$ &$L^\infty$ Err &Ord  & $L^\infty$ Err & Ord& $L^\infty$ Err & Ord & $L^\infty$ Err & Ord \\
\hline
$40$	& $10$  &$3.41$e-$02$ 	& &	$1.87$e-$02$	& &  $2.07$e-$02$ &	&	$1.40$e-$02$	\\

$80$	& $20$ & $1.69$e-$02$	& $1.01$ & $8.08$e-$03$	& $1.21$ & $8.75$e-$02$ &	$1.24$ & $4.88$e-$03$ &	$1.51$  \\

$160$	& $40$ & $9.12$e-$03$ &	$0.89$  & $3.07$e-$03$ &	$1.40$  & $3.35$e-$03$ &	$1.38$	& $1.32$e-$03$ & 	$1.88$  \\

$320$	& $80$ & $7.35$e-$03$ &	$0.31$	& $2.89$e-$03$ &	$0.09$ & $3.16$e-$03$ &	$0.09$ & $2.08$e-$03$ &	$-0.65$ \\
\hline
\hline
$N_x$ & $N_t$ &$L^1$ Err &Ord  & $L^1$ Err & Ord& $L^1$ Err & Ord & $L^1$ Err & Ord \\
\hline
$40$	& $10$ &  $6.06$e-$03$	& & $3.06$e-$03$	& & $2.89$e-$03$	&	&	$2.93$e-$03$		\\

$80$	& $20$ & $1.97$e-$03$ &	$1.62$ & $8.10$e-$04$ &	$1.92$ & $6.85$e-$04$ &	$2.08$ & $5.37$e-$04$ & 	$2.45$   \\

$160$	& $40$ & $6.46$e-$04$ &	$1.61$  & $2.07$e-$04$ &	$1.97$  & $1.68$e-$04$ &	$2.02$	& $6.82$e-$05$ &	 $2.98$ \\

$320$	& $80$ & $2.21$e-$04$ &	$1.55$	 & $6.82$e-$05$ &	$1.60$  & $5.07$e-$05$ &	$1.73$ & $2.58$e-$05$ &	$1.40$ \\
\hline
\end{tabular}
\end{center}
\end{table}
%
% TABLE CON BETA JP
\begin{table}[h!]
\caption{(Example 4.) Local Errors and orders in $L^\infty$ and  $L^1$ norms. \label{ex1n:table2}}
\vspace{-0.3 cm}
\begin{center}
\hspace*{-0.3cm}
\begin{tabular}{c c|c c|c c|c c|c c}
& & \multicolumn{2}{|c|}{\bf F-HC ($5\Delta x$)}&\multicolumn{2}{|c|}{\bf AF-HC}& \multicolumn{2}{|c|}{\bf AF-LW4ord}&\multicolumn{2}{|c}{\bf WENO 2/3}\\
\hline
\hline
$N_x$ & $N_t$ &$L^\infty$ Err &Ord  & $L^\infty$ Err & Ord& $L^\infty$ Err & Ord & $L^\infty$ Err & Ord  \\
\hline
$40$	& $10$ &       $5.52$e-$03$		&	&	$2.59$e-$03$ &	& $4.04$e-$03$	& &	$4.04$e-$03$		\\

$80$	& $20$ & $1.58$e-$03$	& $1.81$ & $4.92$e-$04$	& $2.40$ & $1.37$e-$04$ & $4.88$  & $9.31$e-$04$	& $2.12$ \\

$160$	& $40$ &  $2.97$e-$04$	& $2.41$	& $2.22$e-$04$	& $1.15$ & $1.46$e-$05$	& $3.22$ & $6.06$e-$05$	& $3.94$ \\

$320$	& $80$ &  $8.02$e-$05$	& $1.89$ & $4.33$e-$05$	& $2.36$   & $1.32$e-$07$	& $3.47$	& $2.23$e-$06$	& $4.77$ \\
\hline
\hline
$N_x$ & $N_t$ &$L^1$ Err &Ord  & $L^1$ Err & Ord& $L^1$ Err & Ord & $L^1$ Err & Ord \\
\hline
$40$	& $10$ &   $2.09$e-$03$		&	&	$1.14$e-$03$ 	&	& $6.23$e-$04$	& &	$1.38$e-$03$		\\

$80$	& $20$ &  $4.47$e-$04$	& $2.23$ & $3.07$e-$04$	& $1.91$  & $2.08$e-$05$	& $4.91$ & $2.33e$-$04$	& $	2.57$ \\

$160$	& $40$ &  $9.08$e-$05$	& $2.30$	& $8.06$e-$05$	& $1.93$  & $1.16$e-$06$	& $4.16$ & $1.62$e-$05$	& $3.84$ \\

$320$	& $80$ &  $2.13$e-$05$	& $2.09$ & $2.00$e-$05$	& $2.01$  & $5.13$e-$07$ & $1.18$ & $7.66$e-$07$	& $4.40$ \\
\hline
\end{tabular}
\end{center}
\end{table}

\noindent Looking at Tab. \ref{ex1n:table1} we can note that all the tested schemes suffer a sort of ``saturating'' effect in $L^\infty$ norm,  showing some difficulties in dropping the error in the last refinement, especially in the case of the WENO scheme. On the other hand, the high-order convergence rate in $L^1$ norm testifies the reliability of the schemes also in this situation. The best results in Tab. \ref{ex1n:table1} are clearly given by the WENO scheme, nevertheless, the AF-HC scheme performs better in terms of error and orders in both norms with respect to its basic version with $\varepsilon=5\Delta x$, whereas the fourth order scheme presents the usual behavior in $L^\infty$ norm with respect to the second order AF-HC scheme, however maintaining the same order of errors.  
If instead we look at the errors in regions of regularity reported in Tab.  \ref{ex1n:table2}, we can acknowledge that all the schemes achieve optimal order in both norms, with best results now given by the AF-LW4ord scheme, also with respect to the WENO scheme. 

As already seen in the previous Example 3, in Fig. \ref{ex4:fig2} we can recognize the ability of the AF scheme to better localize the regions of singularity with respect to the basic procedure which uses the monotone scheme way more than necessary. 

\begin{figure}[h!]
	\includegraphics[width=0.48\textwidth]{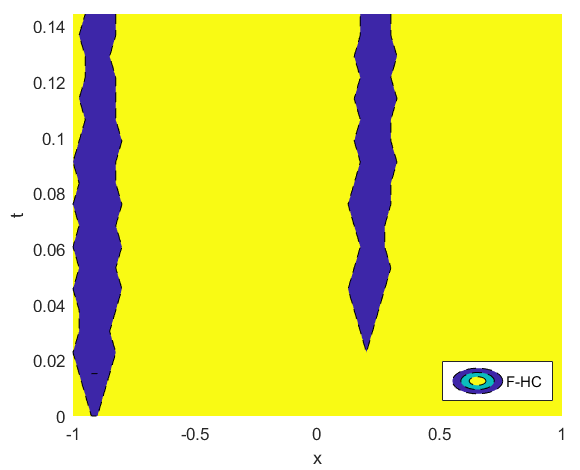}
	\includegraphics[width=0.48\textwidth]{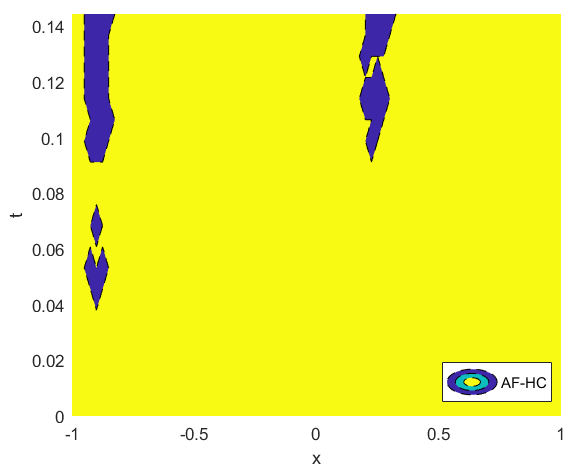}
	\caption{\small{(Example 4.) Regions of activity of $S^M$ (blue) and $S^A$ (yellow) for the F-HC scheme (left) and AF-HC scheme (right) with $\Delta x=0.025$.}\label{ex4:fig2}}
\end{figure}

%%%%%%%%%%%%%%%%%%%%%%%%%%%%%%%%%%%%%%%%%%
%%%%%%%%%%%%%%% 2D EXAMPLE %%%%%%%%%%%%%%%%%%%%%%%%%%%
%%%%%%%%%%%%%%%%%%%%%%%%%%%%%%%%%%%%%%%%%%%%%
{\bf Example 5: Evolution in 2D by dimensional splitting.}
We conclude this section on numerical simulations showing a convenient procedure to solve simple two-dimensional problems by making use of the one-dimensional schemes defined in the previous sections. Let us consider a classical problem similar to the Burgers' equation, which is strictly connected to (\ref{ex3:eq1}),
\begin{equation}\label{ex5:eq1}
\left\{
\begin{array}{l}
v_t+(v_x+1)^2+(v_y+1)^2=0\qquad \textrm{ in } (0,T)\times\Omega, \\
v(0,x,y)=-0.5\left(\cos(\pi x)+\cos(\pi y)\right),
\end{array}
\right.
\end{equation}
with $\Omega=[0,2]^2$ and periodic boundary conditions. As done for problem (\ref{ex3:eq1}), we consider the final time $T=\frac{4}{5\pi^2}$, when the solution is still smooth, and then $T=\frac{3}{2\pi^2}$, time at which an interesting set of singularities develops. 
The exact solution is computed by the \emph{Hopf-Lax} formula,
\begin{equation*}
v(t,x,y)=\left(\min_{a\in A} \frac{1}{2}\cos(x-at)+\frac{1}{4}a^2-a + \min_{b\in A} \frac{1}{2}\cos(y-bt)+\frac{1}{4}b^2-b\right),
\end{equation*}
with $A=[-5,5]$. 

In this situation, since the hamiltonian can be expressed as a sum of one-dimensional hamiltonian, depending on the evolution along the $x$ and $y$ direction, respectively, we can use a dimensional splitting to solve the problem. More precisely, if we write $H(v_x,v_y)=H_1(v_x)+H_2(v_y)$, we can approximate the solution by solving sequentially the problems in one space dimension
$$
v_t+H_1(v_x)=0\qquad\textrm{and}\qquad v_t+H_2(v_y)=0,
$$
keeping each time the other space variable constant. Since the hamiltonians trivially commute, 
we can use the simple \emph{Lie-Trotter} splitting
\begin{equation}
u^{n+1}=S^{\Delta t}_y\left(S^{\Delta t}_x(u^n)\right),
\end{equation}
where $S^{\Delta t}_x$ and $S^{\Delta t}_y$ are numerical schemes of time step $\Delta t$ for the problems in the $x$ and $y$ direction, respectively, without introducing errors in the time evolution. For more details about dimensional splitting techniques we refer the reader to \cite{PaolucciPhD} and the references therein. 

We use the same schemes as in Example 3 and a slightly more restrictive CFL number with respect to problem (\ref{ex3:eq1}) in order to use coarser grids, which is set to $\lambda=\frac{4}{5\pi^2}\approx 0.08$ for the first test, and $\lambda=\frac{3}{4\pi^2}\approx 0.076$ for the latter.

\begin{figure}[h!]
\includegraphics[width=0.5\textwidth]{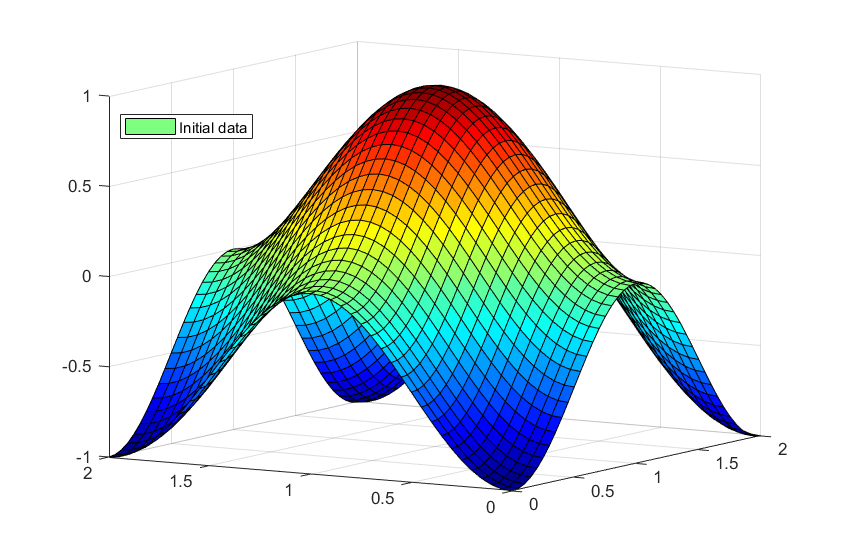}
\includegraphics[width=0.5\textwidth]{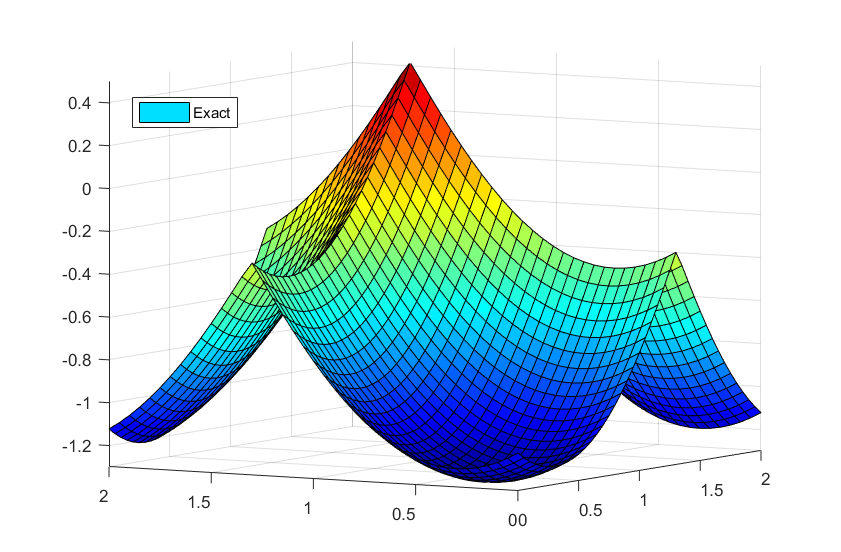}
\includegraphics[width=0.5\textwidth]{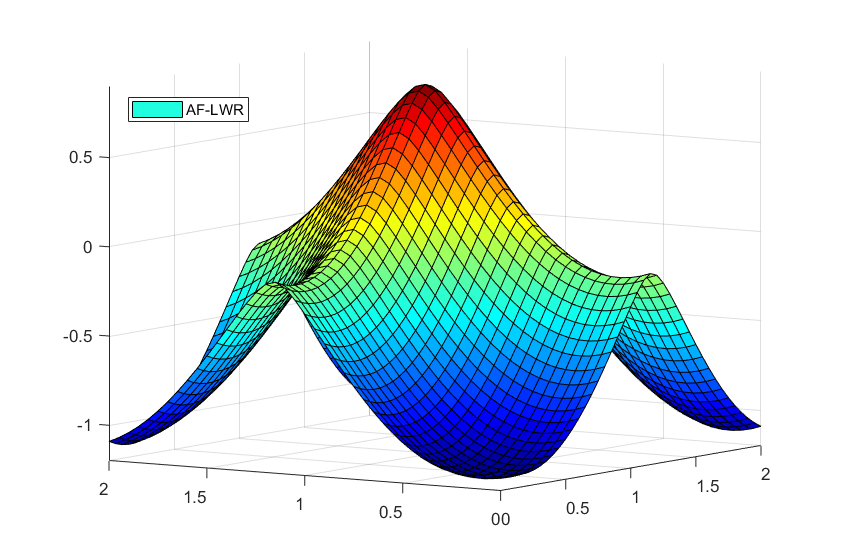}
\includegraphics[width=0.5\textwidth]{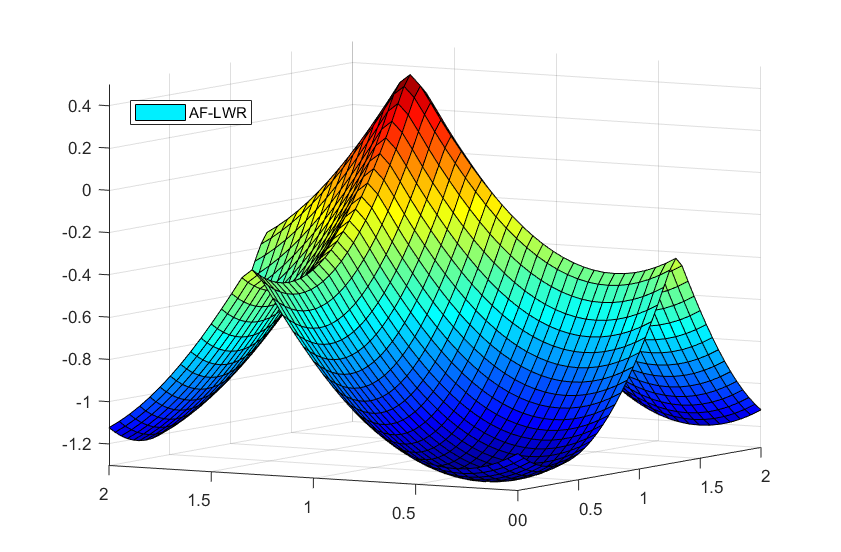}
\caption{\small{(Example 5.)} Top: Initial datum (left) and exact solution at $T=3/(2\pi^2)$ (right). Bottom: solution at $T=4/(5\pi^2)$ (left) and $T=3/(2\pi^2)$ (right) computed by the AF-LWR scheme with $\Delta x=0.1$ \label{ex5:fig1}.}
\end{figure}

\begin{table}[h!]
\caption{(Example 5.) $T=4/(5\pi^2)$. Errors and orders in $L^\infty$ and  $L^1$ norms. \label{ex5:table1}}
\vspace{-0.3 cm}
\begin{center}
\hspace*{-0.3cm}
\begin{tabular}{c c|c c|c c|c c| c c}
& & \multicolumn{2}{|c|}{\bf LWR}&\multicolumn{2}{|c|}{\bf F-LWR ($10\Delta x$)}&\multicolumn{2}{|c|}{\bf AF-LWR}&\multicolumn{2}{|c}{\bf WENO 2/3}\\
\hline
\hline
$N_x$ & $N_t$ &$L^\infty$ Err &Ord  & $L^\infty$ Err & Ord& $L^\infty$ Err & Ord & $L^\infty$ Err & Ord \\
\hline
$20$ &	$10$	& $7.45$e-$02$	&	&	$7.75$e-$02$	& & $9.35$e-$02$ & & $8.66$e-$02$			\\

$40$ &	$20$	& $3.38$e-$02$ &	$1.14$ & $5.12$e-$02$ &	$0.60$ & $3.44$e-$02$ &	$1.44$    & $3.59$e-$02$ &	$1.27$ \\

$80$ &	$40$	& $1.49$e-$02$ &	$1.18$ & $3.25$e-$02$ &	$0.66$ & $5.98$e-$03$ &	$2.52$	& $1.30$e-$02$ &	$1.47$  \\

$160$ &	$80$	& $6.42$e-$03$ &	$1.22$ & $1.94$e-$02$ &	$0.75$	& $1.78$e-$03$ &	$1.75$   & $4.87$e-$03$ &	$1.41$\\
\hline
\hline
$N_x$ & $N_t$ &$L^1$ Err &Ord  & $L^1$ Err & Ord& $L^1$ Err & Ord & $L^1$ Err & Ord \\
\hline
$20$ &	$10$ &  $3.67$e-$02$  &	&	$4.42$e-$02$	& & $4.15$e-$02$	& &  $3.71$e-$02$\\

$40$ &	$20$ & $9.53$e-$03$	& 	$1.94$ & $1.21$e-$02$	& 	$1.87$  & $8.97$e-$03$	& 	$2.21$   & $1.00$e-$02$	& 	$1.89$\\

$80$ &	$40$ & $2.28$e-$03$	& 	$2.06$  & $4.29$e-$03$	& 	$1.49$ & $1.83$e-$03$	& 	$2.29$	& $1.95$e-$03$	& 	$2.36$ \\

$160$ &	$80$ & $6.51$e-$04$	& 	$1.81$ & $1.70$e-$03$	& 	$1.33$	&  $5.23$e-$04$	& 	$1.81$  & $4.50$e-$04$	& $2.11$\\
\hline
\end{tabular}
\end{center}
\end{table}

\begin{table}[h!]
\caption{(Example 5.) $T=3/(2\pi^2)$. Errors and orders in $L^\infty$ and  $L^1$ norms. \label{ex5:table2}}
\vspace{-0.3 cm}
\begin{center}
\hspace*{-0.3cm}
\begin{tabular}{c c|c c|c c|c c| c c}
& & \multicolumn{2}{|c|}{\bf LWR }&\multicolumn{2}{|c|}{\bf F-LWR ($10\Delta x$) }&\multicolumn{2}{|c|}{\bf AF-LWR}&\multicolumn{2}{|c}{\bf WENO 2/3}\\
\hline
\hline
$N_x$ & $N_t$ &$L^\infty$ Err &Ord  & $L^\infty$ Err & Ord& $L^\infty$ Err & Ord & $L^\infty$ Err & Ord \\
\hline
$20$ & $20$ &  $1.69$e-$01$	&	&	$9.49$e-$02$	& & $1.12$e-$01$ & & $8.68$e-$02$			\\

$40$ & $40$ & $6.39$e-$02$ &	$1.40$ & $3.67$e-$02$ &	$1.37$ &$3.66$e-$02$ &	$1.61$    & $2.27$e-$02$ &	$1.93$ \\

$80$ & $80$ & $3.23$e-$02$ &	$0.98$ &$1.41$e-$02$ &	$1.38$ &  $1.46$e-$02$ &	$1.33$	& $9.08$e-$03$ &	$1.32$  \\

$160$ & $160$ & $2.64$e-$02$ &	$0.29$ & $3.73$e-$03$ &	$1.91$ & $3.86$e-$03$ &	$1.92$   & $2.22$e-$03$ &	$2.03$\\
\hline
\hline
$N_x$ & $N_t$ &$L^1$ Err &Ord  & $L^1$ Err & Ord& $L^1$ Err & Ord & $L^1$ Err & Ord \\
\hline
$20$ & $20$ &  $6.18$e-$02$  &	&	$4.62$e-$02$	& & $5.40$e-$02$	& &  $3.60$e-$02$\\

$40$ & $40$ & $1.74$e-$02$	& 	$1.83$ & $8.19$e-$03$	& 	$2.50$  & $8.79$e-$03$	& 	$2.62$   & $4.68$e-$03$	& 	$2.94$\\

$80$ & $80$ & $4.54$e-$03$	& 	$1.94$  & $1.88$e-$03$	& 	$2.13$ & $1.82$e-$03$	& 	$2.27$	& $6.92$e-$04$	& 	$2.76$ \\

$160$ & $160$ & $1.13$e-$03$	& 	$2.00$ & $3.73$e-$04$	& 	$2.33$	&  $3.83$e-$04$	& 	$2.25$  & $7.62$e-$05$ & 	$3.18$\\
\hline
\end{tabular}
\end{center}
\end{table}
%%%%%%%%%%%%%%%%
\begin{table}[h!]
\caption{(Example 5.) $T=3/(2\pi^2)$. CPU times in seconds. \label{ex5:table3}}
\vspace{-0.2 cm}
\centering
\begin{tabular}{c c| p{1.7cm}  | p{1.7cm} | p{1.7cm}  | c }
$N_x$ & $N_t$ & \centering {\bf  LWR } & \centering {\bf F-LWR }& \centering {\bf AF-LWR}& {\bf WENO 2/3}\\
\hline
\hline
$20$ & $20$	&\centering 	$0.002\ s$ 	&\centering 	$0.002\ s$ 	&\centering $0.009\ s$  &$0.010\ s$   \\

$40$ & $40$ &\centering 	$0.015\ s$  &\centering 	$0.016\ s$   &\centering $0.055\ s$  	& $0.054\ s$    \\

$80$ & $80$ &\centering 	$0.108\ s$  &\centering  	$0.127\ s$   &\centering $0.427\ s$   	& $0.324\ s$ \\

$160$ & $160$ &\centering 	$0.867\ s$  &\centering  	$0.902\ s$   &\centering $3.642\ s$   	 & $2.569\ s$ 	\\
\hline
\end{tabular}
\end{table}

As we could expect, in this example we have analogous result with respect to  Example 3, with the AF scheme performing well in both situations and better than the F scheme in the regular case (see Tabs. \ref{ex5:table1}-\ref{ex5:table2}). Here again the simple filtered scheme has slightly better results after the singularities develop, due to the action of the $\phi$ function in the regions of singularity, but the loss of accuracy is in fact minimal. Moreover, our scheme performs as good as the WENO scheme when the solution is still regular, whereas the latter performs  better in the second case. 
Concerning the CPU times, looking at Tab. \ref{ex5:table3} we can see that, differently from the one-dimensional case, the WENO scheme is faster than the AF scheme in the last two refinements, whereas the F-LWR scheme is comparable to the simple LWR high-order scheme and both are three/four times faster than the other two considered schemes. 

%%%%%%%%%%%%%%%%%%%%%%%%%%%%%%%%%%%%%%%%%%%%%%
%%%%%%%%%%%%%%%  SECTION 6: conclusions
%%%%%%%%%%%%%%%%%%%%%%%%%%%%%%%%%%%%%%%%%%%%%%
\section{Conclusions}\label{sec:conclusions}
We have presented a rather simple way to construct convergent schemes coupling a monotone and a high-order scheme via a filter function. A typical feature of filtered schemes is their  high-order accuracy in the regions of regularity for the solution. In fact, the filter function 
can stabilize an otherwise unstable (high-order) scheme, still preserving its accuracy. The main  novelty here is the adaptive and automatic choice of the parameter $\varepsilon^n$ which improves the scheme in \cite{BFS16}. The computation of the switching parameter $\varepsilon^n$, although  more expensive, is still affordable in low dimension. The adaptive  scheme is able to reduce the oscillations which may appear choosing a constant $\varepsilon$ and, as shown by the numerical tests, gives always better results. Finally, we note that the accuracy of adaptive filtered schemes is close to WENO schemes of the same order but filtered schemes are easier to implement, give a rather flexible way to couple different schemes and, as we proved, converge to the viscosity solution. 
A fully 2D	scheme has been used in our recent paper \cite{FPT_arxiv}, in which new formulas for the 2D smoothness indicators are proposed, instead of applying a splitting techniques, for the resolution of the image segmentation problem.  
Thanks to recent computations, we are able to prove a result similar to Prop. \ref{prop_beta} for our new 2D smoothness indicators. Thanks to such a result, we believe that the convergence and high-order consistency could be proven also in fully two-dimensional problems, but this is a future work. 

\subsubsection*{Acknowledgements}
We would like to thank the national group INdAM-GNCS for the financial support given to this research.

%%%%%%%%%%%%%%%%%%%%%%%%%%%%%%%%%%%%%%%%%%%%%%
%%%%%%%%%%%%%%%  BIBLIOGRAPHY
%%%%%%%%%%%%%%%%%%%%%%%%%%%%%%%%%%%%%%%%%%%%%%

%\newpage
%%%%%%%%%%%%%%%%%%%%%%%%%%%%%%%%%%%%%%%%%%%%
%%%%%%%%                                    APPENDICE A
%%%%%%%%%%%%%%%%%%%%%%%%%%%%%%%%%%%%%%%%%%%%%
\clearpage
\appendix
\section{APPENDIX: Technical results}\label{sec:appendix}
For completeness and reader's convenience  we give the proofs of Prop. \ref{prop_beta} and of the properties of the undivided differences and the binomial coefficients involved. This analysis follows the ideas in \cite{ABBM11} where a similar analysis is developed for conservation laws.
\begin{proof}[of Proposition \ref{prop_beta}]
Let us take $r>1$ and without loss of generality, let $x_s=0$ (to simplify the notation). Moreover, we will use the convention $h:=\Delta x$. 
Let us start by reminding that, using the Newton form of the interpolating polynomial, for $k=0,\dots,r-1$ and $j\in \mathbb Z$, we get
\begin{equation}
		P_k(x)=f(x_{j-r+k})+\sum_{i=1}^{r} f[x_{j-r+k},\dots,x_{j-r+k+i}]\omega_{i-1}(x),
		\label{newton_j}
\end{equation}
where $\omega_i(x)=(x-x_{j-r+k})\cdots (x-x_{j-r+k+i})$ and $f[\cdot]$ denotes the divided difference of $f$. 

We proceed with the proof of \emph{\textbf{i)}}. In this case it is sufficient to observe that, since the function $f$ is regular in $\Omega \setminus \{x_s\}$, the properties of the interpolating polynomial directly give
	$$
	P_k^{(l)}(x)=f^{(l)}(x)+O(h^{r+1-l}),\qquad \textrm{ for }x_{j-1}\leq x\leq x_j,\ k=0,\dots,r-1.
	$$
	Moreover, expanding with Taylor, it holds 
	\begin{equation}
	f^{(l)}(x)=f^{(m)}(x_j)O(h)^{m-l}+o(h^{m-l}),
	\end{equation}
	where $m=\max\{s+1,l\}$ and $s=\max\{k:f^{(i)}(x_j)=0,\ \forall i\leq k\}$ ($s\leq r$). Then, integrating (remembering that by hypothesis $s=0 \Rightarrow m=l$), we get
	$$
	h^{2l-3}\int_{x_{j-1}}^{x_j} \left(P_k^{(l)}(x)\right)^2dx= h^{2l-2}\left(f^{(l)}(x_j)\right)^2+o(h^{2l-2}),
	$$
	as we wanted. 	

Let us continue with the proof of \emph{\textbf{ii)}}. In this case the proof is a little more complicated and it is better to treat separately the following two cases:
\begin{enumerate}[(a)]
		\item $0$ is a point of the grid $\{x_i\}$, $i \in \mathbb Z$;
		\item $0\in I_i=(x_{i-1},x_i)$ for some $i \in \mathbb Z$.
\end{enumerate}
\underline{\emph{Case a.}} By hypothesis $0\in \mathcal S_{j+k}$ for at least one $k=0,\dots,r-1$, then, for each fixed $k$, there exists an integer $j_s\in\{k-r+1,\dots,k-1\}$ such that $x_j=-j_sh$ (for $j_s=k-r$ and $j_s=k$ we fall in the case treated previously). Substituting in (\ref{newton_j}),
	$$
	P_k(x)=f((-j_s+k-r)h)+\sum_{i=1}^{r} f[(-j_s+k-r)h,\dots,(-j_s+k-r+i)h]\omega_{i-1}(x),
	$$
with $\omega_i(x)=(x+(j_s-k+r)h)\cdots (x+(j_s-k+r-i)h)$. Moreover, if we define the function $f_h(y):=f(x_j+hy)=f(h(y-j_s))$, 
we can write
\begin{align*}
		f[x_{j-r+k},\dots,x_{j-r+k+i}]&=f[x_j+(k-r)h,\dots,x_j+(k-r+i)h]\\
		&=\frac{f_h[k-r,\dots,k-r+i]}{i! h^i},
\end{align*}
where $f_h[\cdot]$ denotes the undivided difference of $f_h$. 
Now, defining the polynomial
	\begin{equation}
		\label{formula_Qk}
		Q_k(y):=P_k(x_j+hy)=f_h(k-r)+\sum_{i=1}^{r} f_h[k-r,\dots,k-r+i]\frac{q_{i-1}(y)}{i!},
	\end{equation}
	where $q_i(y)=(y-(k-r))\cdots (y-(k-r-i))$, $f_h(y)=f(x_j+hy)$, we can rewrite
	\begin{equation}
		P_k^{(l)}(x)=\frac{d^l}{d x^l}\left(Q_k\left(\frac{x}{h}\right)\right)=\frac{1}{h^l}Q_k^{(l)}(y),\qquad l=1,\dots,r.
		\label{dP_dQ}
	\end{equation}
	Then, applying the change of variable $y=(x-x_j)/h$ in the integral in (\ref{beta}), we have
	\begin{equation}
		h^{2l-3}\int_{x_{j-1}}^{x_j} \left(P_k^{(l)}(x)\right)^2dx= h^{-2}\int_{-1}^0 \left(Q_k^{(l)}(y)\right)^2dy,
		\label{intP_intQ}
	\end{equation}
	where 
	\begin{equation}
		\label{formula_Qkl}
		Q_k^{(l)}(y)=\sum_{i=l}^r f_h[(k-r),\dots,(k-r+i)]\frac{q_{i-1}^{(l)}(y)}{i!}.
	\end{equation}
	At this point, it is useful to notice that (\ref{fh_il2}) for $l=1$ reads, for $i=1,\dots,r$,
	$$
	f_h[k-r,\dots,(k-r+i)]=\sum_{j=0}^{i-1}\binom{i-1}{j}(-1)^{i-j-1}f_h[(k-r+j),(k-r+j+1)].
	$$
	In order to simplify the notation let us call $i_s:=j_s-k+r$, that is to say the index $i_s\in\{1,\dots,r-1\}$ such that $x_j+(k-r+i_s)h=0$.
	Then, by hypothesis, we can write for all $i>t:=\max\{i_s,l-1\}$,
	\begin{align*}
		f_h[k-r,\dots,(k-r+i)]&=\sum_{j=0}^{i_s-1}\binom{i-1}{j}(-1)^{i-j-1}f_h[(k-r+j),(k-r+j+1)]\\
		&\quad +\sum_{j=i_s}^{i-1}\binom{i-1}{j}(-1)^{i-j-1}f_h[(k-r+j),(k-r+j+1)],
	\end{align*}
	and, noticing that for $h\to 0$
	$$
	f_h[z+j_s,z+j_s+1]=h\frac{f((z+1)h)-f(zh)}{h}\to
	\left\{
	\begin{array}{ll}
	hf^\prime(0^+)\qquad&\textrm{ if }z\geq 0\\
	hf^\prime(0^-)&\textrm{ otherwise }
	\end{array}
	\right.
	$$
	we can conclude that
	\begin{align*}
		f_h[k-r,\dots,(k-r+i)]&\approx h\left[\sum_{j=0}^{i_s-1}\binom{i-1}{j}(-1)^{i-j-1}f^\prime(0^-)+\sum_{j=i_s}^{i-1}\binom{i-1}{j}(-1)^{i-j-1}f^\prime(0^+)\right]\\
		&=h\left[f^\prime(0^+)-f^\prime(0^-)\right]\sum_{j=i_s}^{i-1}\binom{i-1}{j}(-1)^{i-j-1}\\
		&=h\left[f^\prime(0^+)-f^\prime(0^-)\right]\binom{i-2}{i_s-1}(-1)^{i-i_s+1}\not =0,
	\end{align*}
	having exploited the relations $\sum_{j=0}^{i}\binom{i}{j}(-1)^{i-j}=0$ and $\sum_{j=0}^{n}\binom{i}{j}(-1)^{i-j}=\binom{i-1}{n}(-1)^{i-n}$, for $0\leq n<i$ by Lemma \ref{lemma1}. 
	Furthermore for $l\leq i\leq i_s$, using the relation \eqref{prop_bin}, we can conclude
	$$
	f_h[k-r,\dots,(k-r+i)]\approx h\sum_{j=0}^{i-1}\binom{i-1}{j}(-1)^{i-j-1}f^\prime(0^-)=0.
	$$
	From what we have done so far we can deduce, recalling that $t:=\max\{i_s,l-1\}$,
	\begin{align*}
		h^{2l-3}\int_{x_{j-1}}^{x_j} \left(P_k^{(l)}(x)\right)^2dx&= h^{-2}\int_{-1}^0\left(\sum_{i=t+1}^r f_h[k-r,\dots,(k-r+i)]\frac{q_{i-1}^{(l)}}{i!}\right)^2\\
		&\approx C_{rk}\left[f^\prime(0^+)-f^\prime(0^-)\right]^2,
	\end{align*}
	where $C_{rk}=\int_{-1}^0\left(\sum_{i=t+1}^r \binom{i-2}{i_s-1}(-1)^{i-i_s+1}\frac{q_{i-1}^{(l)}}{i!}\right)^2$, which is the thesis for Case \emph{a.}
	
\noindent 
\underline{\emph{Case b.}} By hypothesis there exists an integer $j_s\in\{k-r+1,\dots,k\}$ and a number $0<a_s<1$ such that $x_j=(-j_s+a_s)h$. It is clear now that we can repeat the same constructions of the previous case defining the function $f_{a_s,h}(y):=f(h(y-j_s+a_s))$ and using it in place of $f_h$; so, to obtain (\ref{intP_intQ}) it will suffice to apply the change of variables $y=\frac{x}{h}+j_s-a_s$. Then, naming $i_s=j_s-k+r$, for $i\geq t:=\max\{i_s,l\}$,
	\begin{align}
		f_{a_s,h}[k-r,\dots,(k-r+i)]&=\sum_{j=0}^{i_s-2}\binom{i-1}{j}(-1)^{i-j-1}f_{a_s,h}[(k-r+j),(k-r+j+1)] \nonumber\\ 
		&\quad+\binom{i-1}{i_s-1}(-1)^{i-i_s}f_{a_s,h}[j_s-1,j_s] \nonumber \\
		&\quad+\sum_{j=i_s}^{i-1}\binom{i-1}{j}(-1)^{i-j-1}f_{a_s,h}[(k-r+j),(k-r+j+1)],
		\label{f_ah}
	\end{align}
	whence, noticing that
	\begin{align*}
		f_{a_s,h}[j_s-1,j_s]&=f(a_sh)-f((a_s-1)h)\\
		&=a_sh\left(\frac{f(a_sh)-f(0)}{a_sh}\right)+(1-a_s)h\left(\frac{f(0)-f((a_s-1)h)}{(1-a_s)h}\right)\\
		&\approx a_shf^\prime(0^+)+(1-a_s)hf^\prime(0^-)\\
		&=a_sh\left[f^\prime(0^+)-f^\prime(0^-)\right]+hf^\prime(0^-),
	\end{align*}
	and that
	$$
	f_{a_s,h}[z+j_s-1,z+j_s]\to
	\left\{
	\begin{array}{ll}
	hf^\prime(0^+)\qquad&\textrm{ if }z\geq 1\\
	hf^\prime(0^-)&\textrm{ if }z\leq -1,
	\end{array}
	\right.
	$$
	we can infer that if $i=i_s$ (in this case in (\ref{f_ah}) on the right side of the equation we have only the second term), then $f_{a_s,h}[k-r,\dots,(k-r+i)]\approx a_s h\left[f^\prime(0^+)-f^\prime(0^-)\right]\not=0$, whereas %while 
	if $i>i_s$,
	$$
	f_{a_s,h}[k-r,\dots,(k-r+i)]\approx h\left[f^\prime(0^+)-f^\prime(0^-)\right]\left[\binom{i-2}{i_s-1}(-1)^{i-i_s+1}+a_s\binom{i-1}{i_s-1}(-1)^{i-i_s}\right].
	$$
	The last quantity, as it can be easily shown, it is null if and only if $a_s=\frac{i-i_s}{i-1}$; more precisely, for $k$ fixed there exists an integer $i\geq t$ such that $f_{a_s,h}[k-r,\dots,(k-r+i)]\approx Ch\left[f^\prime(0^+)-f^\prime(0^-)\right]$ with $C\not =0$, whence the thesis even in the last case. 
%	\qed
\end{proof}
\begin{lemma}
	\label{lemma1}
	Let us assume $i\geq1$ and write $f[\cdot]$ for the undivided difference of a function $f$. Then, it holds
	\begin{equation}
		f[0,\dots,i]=\sum_{j=0}^{i-l}\binom{i-l}{j}(-1)^{i-l-j}f[j,\dots,j+l],\qquad\textrm{ for }l=0,\dots,i.
		\label{fh_il2}
	\end{equation}
	Moreover, we have that
	\begin{equation}
		\sum_{j=0}^{n}\binom{i}{j}(-1)^{i-j}=\left\{
		\begin{array}{ll}
			\binom{i-1}{n}(-1)^{i-n}\quad&\textrm{ for }n<i\\
			0& \textrm{ for } n=i.
		\end{array}
		\right.
		\label{prop_bin}
	\end{equation}
\end{lemma}
\begin{proof}
	Let us start from the proof of (\ref{fh_il2}) and let us proceed by induction on $i$.
	
	Firstly, let us notice that for $l=i$ the identity is trivially satisfied, whence the case $i=1$ follows directly. Then, for any $l=0,\dots,i-1$, suppose that the statement holds for $i-1$ and for $i>0$ let us compute,
	$$
	\begin{alignedat}{3}
	f[0,\dots,i]&=f[1,\dots,i]-f[0,\dots,i-1]\qquad&&\textrm{by definition of}f[\cdot]\\
	&=\sum_{j=0}^{i-l-1}\binom{i-l-1}{j}(-1)^{i-l-1-j}f[j+1,\dots,j+1+l]&\\
	&\quad-\sum_{j=0}^{i-l-1}\binom{i-l-1}{j}(-1)^{i-l-1-j}f[j,\dots,j+l]&&\textrm{by inductive hyp.}\\
	&=f[i-l,\dots,i]+(-1)^{i-l}f[0,\dots,l]\\
	&\quad +\sum_{j=1}^{i-l-1}\binom{i-l-1}{j-1}(-1)^{i-l-j}f[j,\dots,j+l]&&\\
	&\quad+\sum_{j=1}^{i-l-1}\binom{i-l-1}{j}(-1)^{i-l-j}f[j,\dots,j+l]&& \\
	&=f[i-l,\dots,i]+(-1)^{i-l}f[0,\dots,l]\\
	&\quad+\sum_{j=1}^{i-l-1}\binom{i-l}{j}(-1)^{i-l-j}f[j,\dots,j+l]\qquad\quad&&\binom{n-1}{k-1}+\binom{n-1}{k}=\binom{n}{k} \\
	&=\sum_{j=0}^{i-l}\binom{i-l}{j}(-1)^{i-l-j}f[j,\dots,j+l],&&
	\end{alignedat}
	$$
	as we wanted.
	\begin{remark}
		To simplify the notation we have stated the result for $f[0,\dots,i]$ but the proof clearly holds for $f[k,\dots,k+i]$, $\forall k$. In the second identity of the previous chain we have assumed this fact applying the inductive hypothesis on both terms.
	\end{remark}
	
	Let us focus now on the second relation of the lemma (\ref{prop_bin}) and proceed again by induction, but this time on $n:0\leq n<i$.
	For $n=0$ we have $(-1)^i=(-1)^i$, then the identity holds. Suppose that (\ref{prop_bin}) holds for $n-1<i-1$ and compute
	$$
	\begin{alignedat}{3}
	\sum_{j=0}^n\binom{i}{j}(-1)^{i-j}&=\sum_{j=0}^{n-1}\binom{i}{j}(-1)^{i-j}+\binom{i}{n}(-1)^{1-n}\\
	&=\binom{i-1}{n-1}(-1)^{i+1-n}+\binom{i}{n}(-1)^{i-n}\qquad&\textrm{ by inductive hyp.}\\
	&=\binom{i-1}{n-1}(-1)^{i+1-n}-\left[\binom{i-1}{n}+\binom{i-1}{n-1}\right](-1)^{i+1-n}\\
	&=\binom{i-1}{n}(-1)^{i-n}.
	\end{alignedat}
	$$
	For $n=i$ instead, from what we have just seen we can easily compute
	\begin{align*}
		\sum_{j=0}^i\binom{i}{j}(-1)^{i-j}&=\sum_{j=0}^{i-1}\binom{i}{j}(-1)^{i-j}+(-1)^{i-i}\\
		&=\binom{i-1}{i-1}(-1)^{i-i+1}+1\\
		&=-1+1=0.
	\end{align*}
%	\qed
\end{proof}

\end{document}